\documentclass[journal,twoside,web]{ieeecolor}

\usepackage{amsmath,amsfonts,amssymb}
\usepackage{dsfont}
\usepackage{algorithm}
\usepackage{algpseudocode}
\usepackage{accents}
\usepackage{siunitx}
\usepackage{tikz}
\usepackage{pgfplots}
\pgfplotsset{compat=1.18}
\usepgfplotslibrary{groupplots}
\usepackage{cite}
\usepackage{enumerate}
\usepackage{mathtools}
\usetikzlibrary{external}
\usepackage{tabularx}
\usepackage{generic}
\usepackage{graphicx}

\def\BibTeX{{\rm B\kern-.05em{\sc i\kern-.025em b}\kern-.08em
		T\kern-.1667em\lower.7ex\hbox{E}\kern-.125emX}}
	
\renewcommand{\baselinestretch}{0.95}

\newcommand*{\vm}[1]{\boldsymbol{#1}}
\newcommand*{\trans}{^{\top}}

\newcommand*{\nablax}{\nabla_{\vm x_i}}
\newcommand*{\nablaxx}{\nabla_{\vm x_i \vm x_i}^2}
\newcommand*{\nablaxj}{\nabla_{\vm x_j}}
\newcommand*{\neighs}{j \in \mathcal{N}_i}

\newcommand*{\agents}{i \in \mathcal{V}}
\newcommand*{\Ni}[1]{\vm{#1}_{\mathcal N_i}}
\newcommand*{\Nj}[1]{\vm{#1}_{\mathcal N_j}}
\newcommand*{\inds}{^{*}}
\newcommand*{\indqp}{^{q-1}}
\newcommand*{\indq}{^{q}}
\newcommand*{\st}{\operatorname{s.\!t.}}
\newcommand*{\smallspace}{\,\,}
\newcommand*{\indqn}{^{q+1}}

\DeclareMathOperator*{\diag}{\mathrm{diag}}
\DeclareMathOperator*{\blkdiag}{\mathrm{blkd}}

\DeclareMathOperator*{\sign}{\mathrm{sign}}
\newtheorem{theorem}{Theorem}
\newtheorem{lemma}{Lemma}
\newtheorem{assumption}{Assumption}
\newtheorem{corollary}{Corollary}
\newtheorem{remark}{Remark}

\newtheorem{example}{Example}

\newenvironment{contexample}{
	\addtocounter{example}{-1} \begin{example}[cont.]}{
\end{example}}

\newtheorem{assumptionalt}{Assumption}[assumption]

\newenvironment{assumptionp}[1]{
  \renewcommand\theassumptionalt{#1}
  \assumptionalt
}{\endassumptionalt}

\definecolor{mycolor1}{rgb}{0.00000,0.44700,0.74100} %blue 
\definecolor{mycolor2}{rgb}{0.85000,0.32500,0.09800} %red 
\definecolor{mycolor3}{rgb}{0.9290 0.6940 0.1250} %yellow 
\definecolor{mycolor4}{rgb}{0.4940 0.1840 0.5560} %purple 
\definecolor{mycolor5}{rgb}{0.4660 0.6740 0.1880} %green 
\definecolor{mycolor6}{rgb}{0.3010 0.7450 0.9330} %light blue 
\definecolor{mycolor7}{rgb}{0.6350 0.0780 0.1840} %dark red 

\pgfplotscreateplotcyclelist{errorcurves}{ % Reihenfolge für Linienstyles
	{mycolor1,thick},
	{mycolor2,thick},
	{mycolor4,thick, dashed},
	{mycolor3,thick}
}

\pgfplotscreateplotcyclelist{comparison}{ % Reihenfolge für Linienstyles
	{mycolor1,thick},
	{mycolor2,thick},
	{mycolor3,thick},
	{mycolor4,thick},
	{mycolor5,thick},
	{mycolor6,thick},
	{mycolor7,thick}
}

\pgfplotsset{label style={font=\footnotesize},ticklabel style={font=\footnotesize},legend style={font=\tiny}, legend cell align=left}

\pgfplotsset{select coords between index/.style 2 args={
		x filter/.code={
			\ifnum\coordindex<#1\fi
			\ifnum\coordindex>#2\fi
		}
}}

\def\BibTeX{{\rm B\kern-.05em{\sc i\kern-.025em b}\kern-.08em
T\kern-.1667em\lower.7ex\hbox{E}\kern-.125emX}}
\begin{document}
\title{Enforcing Convergence in Sensitivity-Based Distributed Programming via Transformed Primal-Dual Updates}
\author{Maximilian Pierer von Esch, Andreas V\"olz, and Knut Graichen, \IEEEmembership{Senior Member, IEEE}
\thanks{This work is funded by the Deutsche Forschungsgemeinschaft (DFG, German Research Foundation) under project no. 464391622.}
\thanks{The authors are with the Chair of Automatic Control, Friedrich-Alexander-Universität Erlangen-Nürnberg (FAU), Erlangen, Germany.
Email: \{maximilian.v.pierer, andreas.voelz, knut.graichen\}@fau.de}}
\maketitle
\vspace{-4mm}
	%
%%%%%%%%%%%%%%%%%%%%%%%%%%%%%%%%%%%%%%%%%%%%%%%%%%%%%%%%%%%%%%%%%%%%%%%%%%%%%%%%%%%%%%%%%%%%%%%%%%%%%%%%%%%%%%%
	%
\begin{abstract}
Sensitivity-based distributed programming (SBDP) is a decomposition method for solving large-scale nonlinear programs over graph-structured networks. However, its convergence depends on the strength and structure of subsystem coupling. To address this limitation, we propose SBDP+, a distributed optimization scheme based on a structured primal-dual operator design. The method employs first-order sensitivities and primal decomposition to construct low-dimensional local subproblems that are solved in parallel using neighbor-to-neighbor communication. In contrast to SBDP, SBDP+ introduces a novel primal-dual update that ensures convergence under general coupling structures. Specifically, we establish local linear convergence for non-convex problems under standard regularity conditions. Numerical experiments demonstrate the effectiveness of SBDP+ and highlight improved robustness compared to SBDP and representative distributed optimization methods in applications such as statistical learning.
	\end{abstract}
	%
	%%%%%%%%%%%%%%%%%%%%%%%%%%%%%%%%%%%%%%%%%%%%%%%%%%%%%%%%%%%%%%%%%%%%%%%%%%%%%%%%%%%%%%%%%%%%%%%%%%%%%%%%%%%%%%%
	%
\vspace{-2mm}
\section{Introduction}
\label{sec:intro}

% Motivation
Large-scale nonlinear programs (NLP) arise in a wide range of applications, including statistical learning with high-dimensional data or features~\cite{Boyd}, economical optimization problems~\cite{Woolridge}, electrical power systems~\cite{Molzahn}, and distributed nonlinear model predictive control (DMPC)~\cite{Christofides}. These problems often feature separable objectives and structured couplings, which distributed optimization techniques can exploit by decomposing the global problem into smaller, parallelizable subproblems. This enables local computation with limited communication, offering significant advantages over centralized methods in terms of scalability, flexibility, and robustness.
%

% Literature review
While distributed optimization is well understood for convex problems and has lead to the development of many algorithms, e.g., distributed projected gradient descent~\cite{Xi}, dual decomposition \cite{Everett}, distributed non-smooth Newton methods~\cite{Frasch} or the alternating direction method of multipliers~(ADMM)~\cite{Boyd}, many practically relevant applications are inherently non-convex. 
Existing distributed methods for non-convex NLPs often distribute the internal computations of centralized nonlinear programming schemes such as sequential quadratic programming (SQP), interior-point, or augmented Lagrangian methods. Representative examples include distributed SQP \cite{Stomberg}, distributed interior point methods \cite{Engelmann3}, and ALADIN-type schemes \cite{Houska,Engelmann}. While these duality-based approaches can achieve strong convergence properties, they typically rely on consensus reformulations, local copies, or centralized coordination steps, which increase communication and coordination complexity and may limit scalability.
In contrast, primal decomposition methods iterate directly on the shared variables, avoiding the need for strong duality and resulting in smaller subproblem dimensions. Existing approaches include Jacobi-type iterations~\cite{Doan}, forward-backward splitting schemes~\cite{Tang}, and Schwarz-type decomposition methods \cite{Shin}.
%

% SBDP
Sensitivity-based distributed programming (SBDP) recently emerged as a promising primal-decomposition-based approach for non-convex NLPs. Instead of distributing a centralized solver, SBDP directly decomposes the central NLP and coordinates neighboring subproblems through exchanged first-order sensitivities. This yields low-dimensional local NLPs, purely neighbor-to-neighbor communication, and avoids consensus variables or centralized coordination. Prior work demonstrates its effectiveness in linear-quadratic settings~\cite{Scheu,Schneider2,Schneider3}, optimal control~\cite{Pierer3}, and general NLPs~\cite{Pierer}. However, the convergence of SBDP relies on structural assumptions on the coupling, such as a generalized diagonal dominance condition (GDDC)~\cite{Pierer,Scheu} or sufficiently weak interactions, e.g., induced by short prediction horizons in DMPC~\cite{Pierer3}. When these conditions are violated, neither step-size selection nor proximal regularization alone can guarantee convergence. Consequently, the applicability of SBDP remains restricted for strongly coupled systems.

%Contribution and structure
Motivated by these limitations, we propose SBDP+, a sensitivity-based decomposition method that removes these coupling restrictions through a structured primal-dual transformation of the underlying fixed-point iteration. Unlike SBDP, which relies on a contraction induced by the problem structure, SBDP+ enforces local contractivity directly at the operator level and thereby avoids GDDC-type assumptions. We establish local linear convergence for general non-convex NLPs under standard NLP regularity assumptions. Hereby, non-convexity is understood as in e.g. \cite{Houska,Engelmann,Stomberg}, that is, we assume the existence of regular local minima, while allowing the objective and constraint functions to be non-convex.

The proposed method retains the main advantages of SBDP: low-dimensional local NLPs, neighbor-to-neighbor communication without local copies or centralized coordination and, unlike Schwarz-type methods \cite{Shin}, decomposition directly on the original graph structure, while substantially enlarging the class of NLPS for which convergence can be guaranteed. Furthermore, the method first decomposes the NLP rather than distributing the computations of a centralized solver, reducing communication and allowing for different local solvers. In addition, we derive extensions under relaxed second-order conditions and characterize the trade-off between regularity assumptions and communication requirements.

The remainder of the paper is organized as follows. Section~\ref{sec:ProblemStatement} introduces the problem formulation. Section~\ref{sec:Algorithmic_development} presents SBDP+, while Section~\ref{sec:Algorithmic_Analysis} provides the convergence analysis. Extensions under relaxed second-order assumptions are discussed in Section~\ref{sec:Alg_extensions}, numerical results are presented in Section~\ref{sec:Numerical_Evaluation}, and conclusions are drawn in Section~\ref{sec:Conclusion}.

% Notation 
Notation: For a vector $\vm v \in \mathbb{R}^n $ and matrix $\vm M \in \mathbb{R}^{n\times m}$, the notations $[\vm v]_i$ and $[\vm M]_i$ refer to the $i$-th component and $i$-th row of $\vm v$ and $\vm M$, respectively. Given an ordered index set $\mathcal{S} \subset \mathbb{N}$, $[\vm M]_{\mathcal{S}}$ denotes the sub-matrix of rows $[\vm M]_i$, $i\in \mathcal{S}$ and $[\vm v_i]_{i \in \mathcal{S}}$ denotes the stacked vector of all $\vm v_i \in \mathbb{R}^{n_i}$. Norms without subscript, i.e., $\|\cdot\|$, refer to the Euclidean or induced spectral norm. The maximum and minimum eigenvalues of a square matrix are denoted by $\bar{\lambda}(\cdot)$ and $\underline \lambda(\cdot)$. The integer set from $0$ to $N$ is $\mathbb{N}_{[0,N]}$. An open $r$-neighborhood of a point $\vm v_0 \in \mathbb{R}^n$ is defined as $\mathcal{B}_r(\vm v_0):=\{ \vm v \!\in \mathbb{R}^n\,|\, \|\vm v - \vm v_0\| < r\}$. For a vector-valued function $\vm f(\vm x) : \mathbb{R}^n \rightarrow \mathbb{R}^m$, the Jacobian is $\nabla \vm f(\vm x) = [\nabla f_1(\vm x), \dots, \nabla f_m(\vm x)]\trans \in \mathbb{R}^{m\times n}$, whereby $\nabla f_i$ is the gradient of the $i$-th component and $\nabla\trans f_i$ its transpose. 
%%%%%%%%%%%%%%%%%%%%%%%%%%%%%%%%%%%%%%%%%%%%%%%%%%%%%%%%%%%%%%%%%%%%%%%%%%%%%%%%%%%%%%%%%%%%%%%%%%%%%%%%%
%
%%%%%%%%%%%%%%%%%%%%%%%%%%%%%%%%%%%%%%%%%%%%%%%%%%%%%%%%%%%%%%%%%%%%%%%%%%%%%%%%%%%%%%%%%%%%%%%%%%%%%%%%% 
\section{Problem Statement}
\label{sec:ProblemStatement}
We consider NLPs structured over a given, undirected, connected graph $\mathcal{G}=(\mathcal{V}, \mathcal{E})$, where the node set $\mathcal{V} = \{1, \dots, M\}$ represents a collection of subsystems, referred to as agents. The edge set $\mathcal{E} \subset \mathcal{V} \times \mathcal{V}$ encodes the coupling between them.
%central NLP 
The agents cooperatively solve the central NLP
\begin{subequations}\label{eq:central_NLP}
	\begin{alignat}{3}
		\min_{\vm x_1,\dots,\vm x_M} &\quad  \sum_{i\in\mathcal V} f_i(\vm x_i, \Ni{x}) \label{eq:central_costFunction}\\
		~\st \quad&\quad \vm  g_i( \vm x_i,\Ni{x}) = \vm 0 & \quad|\smallspace&\vm \lambda_i\,, & \quad& \forall \agents \label{eq:central_equality} \\
		&\quad \vm h_i(\vm x_i, \Ni{x})\leq \vm 0 & \quad|\smallspace &\vm \mu_i\,, & \quad& \forall \agents \label{eq:central_inequality}
	\end{alignat}
\end{subequations}
with the local decision variables $\vm x_i \in \mathbb{R}^{n_i}$, $\agents$. The notation after the constraints in \eqref{eq:central_NLP} highlights that the quantities  $\vm \lambda_i \in \mathbb{R}^{n_{gi}}$ and $\vm \mu_i \in \mathbb{R}^{n_{hi}}$ represent the Lagrange multipliers associated with the constraints \eqref{eq:central_equality} and \eqref{eq:central_inequality}, respectively. The set ${\mathcal{N}_i := \{j \in \mathcal{V}\,|\,(i,j) \in \mathcal{E}, i \neq j\}}$ denotes the neighbors $ j \in \mathcal{V}$ that are directly coupled with agent $\agents$. These couplings may arise in the objective function \eqref{eq:central_costFunction} and/or in the (in)equality constraints \eqref{eq:central_equality} -- \eqref{eq:central_inequality} through the neighboring decision variables $\vm x_j \in \mathbb{R}^{n_j}$ for each $\neighs$. The stacked notation $\Ni{x}:= [\vm x_j]_{\neighs}$ collects all neighboring variables. Accordingly, each agent minimizes a local objective function $f_i:\mathbb{R}^{n_i}\times\mathbb{R}^{n_{\mathcal{N}_i}} \rightarrow \mathbb{R}$ subject to (coupled) equality constraints $\vm g_i:\mathbb{R}^{n_i}\times\mathbb{R}^{n_{\mathcal{N}_i}} \rightarrow \mathbb R^{n_{gi}}$ and (coupled) inequality constraints $\vm h_i:\mathbb{R}^{n_i}\times\mathbb{R}^{n_{\mathcal{N}_i}} \rightarrow \mathbb R^{n_{hi}}$ with $n_{\mathcal{N}_i}:=\sum_{\neighs} n_j$. All functions appearing in the central NLP \eqref{eq:central_NLP} are assumed to be at least three times continuously differentiable. 
We define the central Lagrangian of problem~\eqref{eq:central_NLP} as
\begin{align}
	L(\vm x, \vm \lambda, \vm \mu) = \sum_{\agents}L_i(\vm x_i, \vm \lambda_i, \vm \mu_i, \Ni{x})\,, \label{eq:central_Lagrangian}
\end{align}
with the local Lagrangians $L_i = L_i(\vm x_i, \vm \lambda_i, \vm \mu_i,\Ni{x})$ given by
\begin{align}
	L_i:= f_i(\vm x_i, \Ni{x}) + \vm \lambda_i\trans \vm  g_i( \vm x_i,\Ni{x}) + \vm \mu_i\trans  \vm h_i(\vm x_i, \Ni{x}) \label{eq:local_Lagrangian}
\end{align}
for every $\agents$.  The centralized decision variable is denoted by $\vm x =[\vm x_i]_{\agents} \in \mathbb R^{n}$, and the stacked multipliers by $ \vm \lambda =  [\vm \lambda_i]_{\agents} \in \mathbb{R}^{n_g}$ and $\vm \mu = [\vm \mu_i]_{\agents} \in \mathbb{R}^{n_h}$. These are jointly represented in the primal-dual solution vector  $ \vm p:=[\vm x\trans,\, \vm \lambda\trans,\, \vm \mu\trans ]\trans \in \mathbb{R}^p$ of NLP \eqref{eq:central_NLP} with total dimension $p=n+n_g+n_h$. The centralized constraints are $\vm g(\vm x):=[\vm g_i(\vm x_i, \Ni{x})]_{\agents}$ and $\vm h(\vm x):=[\vm h_i(\vm x_i, \Ni{x})]_{\agents}$. In large-scale systems with many agents, the centralized NLP~\eqref{eq:central_NLP} becomes high-dimensional and computationally challenging. Therefore, the focus is on developing a distributed solution approach that exploits the underlying graph structure and relies only on neighbor-to-neighbor communication.

\section{Algorithmic development}
\label{sec:Algorithmic_development}
The following section develops SBDP+ for solving structured NLPs \eqref{eq:central_NLP} in a distributed setting. The key difficulty addressed is that standard SBDP updates do not, in general, induce a contractive primal-dual iteration under arbitrary coupling, since the associated fixed-point operator may become unstable in the absence of additional structural assumptions. To overcome this limitation, we introduce a modified primal-dual update scheme based on the solution of sensitivity-augmented local NLPs. This yields a structured transformation of the resulting primal-dual search directions, designed to enforce convergence. First, we derive the modified local NLPs used to compute primal-dual search directions, then introduce the corresponding primal-dual update, and finally present the complete distributed algorithm. 

\subsection{Modified local NLPs}
Each agent $\agents$ constructs a decoupled, local NLP in terms of a local search direction $\vm s_i \in \mathbb{R}^{n_i}$ which is subsequently solved in each iteration $q=1,2,\dots$ of SBDP+
	\begin{subequations}\label{eq:local_NLP}
		\begin{alignat}{2}
			\min_{\vm s_i} \quad & \bar f_i\indq(\vm s_i) \label{eq:local_costFunction}
			\\  \st \quad&  \vm{\bar g}_i\indq(\vm s_i)= \vm 0& \quad|\smallspace& \vm \nu_{i} \label{eq:local_equality}  \\
			 & \vm{\bar h}_i\indq(\vm s_i)\leq\vm 0 &\quad|\smallspace& \vm \kappa_{i} \label{eq:local_inequality}
		\end{alignat}
	\end{subequations}
	with the modified local cost functions
	\begin{equation} \label{eq:definition_costFunction}
	\bar f_i^{q}(\vm s_i) := f_i(\vm x_i\indq + \vm s_i, \Ni{x}\indq) +\frac{\rho_i}{2}\|\vm s_i\|^2\! +\! \sum_{\neighs}\!\! \nablax\trans L_j\indq \vm s_i
	\end{equation}
	and local equality and inequality constraints
	\begin{equation}\label{eq:definition_constraints}
		\vm{\bar g}_i\indq(\vm s_i) := \vm g_i( \vm x_i\indq + \vm s_i,\Ni{x}\indq),\smallspace \vm{\bar h}_i\indq(\vm s_i):= \vm h_i( \vm x_i\indq + \vm s_i,\Ni{x}\indq)
	\end{equation}
	 which depend explicitly on the search direction $\vm s_i$ and implicitly on the value of $\vm p$ at iteration $q$ indicated by the superscript. Similar to \eqref{eq:central_NLP}, the notation in~\eqref{eq:local_NLP} emphasizes that $\vm \nu_{i} \in \mathbb R^{n_{g_i}}$ and $\vm \kappa_{i} \in \mathbb R^{n_{h_i}}$ are local Lagrange multipliers for the constraints \eqref{eq:local_equality} and \eqref{eq:local_inequality}. Let $ \vm y_i := [\vm s_i\trans, \vm \nu_i\trans, \vm \kappa_i\trans ]\trans \in \mathbb{R}^{p_i}$, $p_i = n_i + n_{gi} + n_{hi}$ be the primal-dual solution of \eqref{eq:local_NLP} and 
	 \begin{equation} \label{eq:local_NLP_Lagrangian}
		\bar{L}_i\indq(\vm y_i) = \bar f_i\indq(\vm s_i) + \vm \nu_i\trans \vm{\bar g}_i\indq(\vm s_i) + \vm \kappa_i\trans \vm{\bar h}_i\indq(\vm s_i)
	 \end{equation}
	the Lagrangian of the local NLPs \eqref{eq:local_NLP} at some iteration~$q$. The couplings are resolved via a primal decomposition approach which is realized by treating the neighboring optimization variables $\vm x_j$, $\neighs$ as fixed at their value of the current iteration $q$. The objective~\eqref{eq:definition_costFunction} combines three parts: The first term represents the agent's local objective, in direction $\vm s_i$, the second is quadratic regularization term $\frac{\rho_i}{2}\|\vm s_i\|^2$ where $\rho_i\in \mathbb{R}_{\geq0}$ is a suitable penalty parameter, and the third are sensitivity terms. They collectively account for the first-order influence that a step in direction $\vm s_i$ has on the neighboring agents' optimal value function. Specifically, it is defined as the corresponding directional derivative of the neighbors' local Lagrangian $L_j(\cdot)$, $\neighs$, given by the gradient of the current iteration $ \nablax L_j\indq:= \nablax L_j(\vm x_j\indq, \vm \lambda_j\indq, \vm \mu_j\indq, \Nj{x}\indq)$, in direction of $\vm s_i$. The required gradient $\nablax L_j(\cdot)$ is computed in straightforward fashion from \eqref{eq:local_Lagrangian} as
	\begin{align}\label{eq:gradient_withoutneighboraffine}
		& \nablax L_j(\vm x_j, \vm \lambda_j, \vm \mu_j, \Nj{x}) =  \nablax f_j(\vm x_j, \Nj{x}) \nonumber                        \\
		& \phantom{=}+ \nablax\trans \vm g_j(\vm x_j, \Nj{x}) \vm \lambda_j + \nablax\trans \vm h_j(\vm x_j, \Nj{x}) \vm \mu_j\,,
	\end{align}
where $\vm \lambda_j$ and $\vm \mu_j$ are the Lagrange multipliers of neighbor $\neighs$. Note that $L_j(\cdot)$, $f_j(\cdot)$, $\vm g_j(\cdot)$, and $\vm h_j(\cdot)$ in \eqref{eq:gradient_withoutneighboraffine} explicitly depend on $\vm x_i$ via $\Nj{x}$. However, computing \eqref{eq:gradient_withoutneighboraffine} may involve decision variables of second-order neighbors of agent $i$, i.e., $\mathcal{N}_i^{2} := (\cup_{\neighs}  \mathcal{N}_j) \setminus(\mathcal{N}_i \cup \{i\})$, which are typically not available in a neighbor-to-neighbor communication network. Thus, each agent $\agents$ computes the mirroring gradient $\nablaxj L_i\indq$ instead and sends it to the respective neighbors $\neighs$. 
\subsection{Relation to SBDP and standard update limitation}
	When, after solving \eqref{eq:local_NLP}, the full-step update 
\begin{align} \label{eq:SBDP_update}
	\vm x_i\indqn = \vm x_i\indq + \vm s_i\indq\,,\quad \vm \lambda_i\indqn = \vm \nu_i\indq\,, \quad \vm \mu_i\indqn = \vm \kappa_i\indq
\end{align}	
is applied, the recently proposed SBDP method \cite{Pierer,Scheu} is recovered. However, its convergence depends on the coupling structure between agents as shown in the following example. 
	\begin{example} \label{ex:1}
		\textit{
	Consider the optimization problem 
	\begin{alignat}{1} \label{eq:example_NLP1}
		\min_{x_1,\, x_2} \quad 0.5x_1^2 + 0.5x_2^2 \quad \st \smallspace x_1 + ax_2 = 0 
	\end{alignat}
which is of the form \eqref{eq:central_NLP} with $f_{i}(x_i,x_j) = 0.5x_i^2$, $i\in \{1,\,2\}$, $g_1(x_1,x_2)= x_1 +ax_2$, and coupling parameter $a \in \mathbb{R}$. Explicitly solving the NLPs \eqref{eq:local_NLP} for some $\rho_i=\rho$, $\forall \agents$ and applying the full-step update, leads to the following recursion
\begin{align} \label{eq:discrete_time_exampleSystem}
	\begin{bmatrix}
		x_1\indqn \\ 
		x_2\indqn \\
		\lambda_1\indqn
	\end{bmatrix} = \begin{bmatrix}
		0 &-a &0 \\ 0 & \frac{\rho}{1+\rho}&- \frac{a}{1+\rho}\\ \rho & (1+\rho)a & 0
	\end{bmatrix}\begin{bmatrix}
		x_1\indq \\ 
		x_2\indq \\
		\lambda_1\indq
	\end{bmatrix}
\end{align}
with eigenvalues $ \zeta_1 = \frac{\rho}{\rho +1} $ and $ \zeta_{2/3} = \pm ja$. Thus, the recursion is asymptotically stable if and only if $|a|<1$ and cannot be stabilized by any choice of $\rho$. 
}
\end{example}
Example \ref{ex:1} shows that divergence caused by strong coupling cannot, in general, be removed by tuning the regularization parameter $\rho_i$.
Moreover, the simple damped update
\begin{align} \label{eq:damped_SBDP_update}
\vm x_i\indqn \!=\! \vm x_i\indq \!+\! \alpha\vm s_i\indq,\, [\vm \lambda_i\indqn\!, \vm \mu_i\indqn] \!=\! [\vm \lambda_i\indq, \vm \mu_i\indq] \!+\! \alpha [\vm s_{\lambda,i}\indq, \vm s_{\mu,i}\indq]
\end{align}	
with step size $\alpha\in \mathbb{R}_{>0}$ and $\vm s_{\lambda,i}\indq := \vm \kappa_i\indq - \vm \lambda_i\indq $, $\vm s_{\mu,i}\indq := \vm \nu_i\indq - \vm \mu_i\indq $ does generally not restore convergence either, as demonstrated in the next example.
\begin{contexample}
\textit{
By adding the constraint $g_2(x_2,x_1)= x_1 + x_2 = 0 $ to \eqref{eq:example_NLP1}, one finds that for $a=4$ the eigenvalues of the recursion \eqref{eq:discrete_time_exampleSystem} are $ \zeta_{1/2} = 1+ \alpha $, $ \zeta_{3/4} = 1 - 3\alpha$ which are independent of $\rho_i$. Thus, no choice of $\alpha$ or $\rho_i$, {$\forall \agents$} exists to enforce convergence in this case.}
\end{contexample}

\subsection{Primal-dual update}
These examples reveal that divergence is caused by the structure of the induced primal-dual recursion itself rather than insufficient regularization or damping. To overcome this limitation, SBDP+ modifies the primal-dual update through a structured matrix-valued transformation. Rather than directly substituting the local NLP solution, e.g. \eqref{eq:SBDP_update}, the update applies a local preconditioned correction designed to stabilize the linearized SBDP-recursion. To this end, after computing the local search direction $\vm s_i\indq$ and multipliers $\vm \nu_i\indq$ and $\vm \kappa_i\indq$ from \eqref{eq:local_NLP}, each agent $\agents$ updates its primal-dual variables as
\begin{align} \label{eq:primal_dual_update}
\vm p_i\indqn = \vm p_i \indq + \alpha \vm P_i\indq(\vm y_i\indq) (\vm y_i\indq - \vm d_i(\vm p_i\indq))
\end{align}
with the step size $\alpha\in(0,1)$, $\vm p_i:= [\vm x_i\trans, \vm \lambda_i\trans, \vm \mu_i\trans]\trans \in \mathbb{R}^{p_i}$, and offset $\vm d_i(\vm p_i) := [\vm 0\trans, \vm \lambda_i \trans, \vm \mu_i\trans]\trans \in \mathbb{R}^{p_i}$. The matrix-valued mapping $\vm P_i\indq : \mathbb{R}^{p_i} \rightarrow \mathbb{R}^{p_i \times p_i}$ is given as
\begin{align} \label{eq:MixingMatrix}
\vm P_i\indq(\vm y_i) \!\!=\!\! \begin{bmatrix}\!
\nabla_{\vm s_i \vm s_i}^2 \bar{L}_i\indq( \vm y_i) \!&\!\! \nabla_{\vm s_i}\trans \vm{\bar g}_i\indq(\vm s_i) \!\!&\!\! \nabla_{\vm s_i}\trans \vm{\bar h}_i\indq(\vm s_i)\\
-\beta \nabla_{\vm s_i} \vm{\bar g}_i\indq(\vm s_i) \!&\! \vm 0 \!&\! \vm 0 \\
-\beta \vm K_i \nabla_{\vm s_i} \vm{\bar h}_i\indq(\vm s_i) \!\!\!\!\!&\! \vm 0 \!&\! -\beta \vm{\bar H}_i\indq(\vm s_i)
\end{bmatrix}
\end{align}
for each $\agents$. Hereby, $\beta \in \mathbb{R}_{>0}$ describes an additional step size for the dual updates, $\vm K_i:= \diag([\kappa_{1,i},\dots,\kappa_{n_{hi},i}])$ is a diagonal matrix consisting of each $\kappa_{k,i}$ in $\vm \kappa_i$, $k \in\mathbb{N}_{[1,n_{hi}]}$, and $\vm{\bar H}_i\indq(\vm s_i)= \diag(\vm{\bar h}_i\indq(\vm s_i))$. The matrix $\vm P_i\indq$ acts as a structured local preconditioner for the primal-dual recursion which can be interpreted locally as an inexact Newton-type step on the KKT conditions of \eqref{eq:central_NLP}, cf. Section \ref{sec:Algorithmic_Analysis}. Its structure is chosen to modify the Jacobian of the linearized SBDP-recursion such that its spectrum can be shifted into the stability region by an appropriate choice of $\alpha$, even when scalar step-size damping alone is insufficient. Since $\vm P_i\indq(\cdot)$ depends only on local and neighboring information, this design remains compatible with neighbor-to-neighbor communication. Returning to Example \ref{ex:1}, we observe that this modification enforces convergence. 
\begin{contexample}
\textit{
Suppose now we apply \eqref{eq:primal_dual_update} with $\beta =1$. Then, we find the eigenvalues of the resulting recursion to be $\zeta_{1/2} = 1 - \frac{\alpha}{2} \pm j \frac{\alpha}{2} c_1(a)$ and $\zeta_{3/4} = 1 - \frac{\alpha}{2} \pm j \frac{\alpha}{2} c_2(a)$ with problem specific constants $c_1(a),c_2(a) \in \mathbb{R}_{>0}$ which depend on a. This shows that if $\alpha < \min\{\frac{4}{1 + c_1(a)^2},\frac{4}{1 + c_2(a)^2}\}$, the recursion is stable under the update \eqref{eq:primal_dual_update} for any $a \in \mathbb{R}$. 
}
\end{contexample} 
This example shows that a sufficiently small step size $\alpha$, combined with \eqref{eq:primal_dual_update}, ensures convergence of SBDP+, a result that the convergence analysis in Section \ref{sec:Algorithmic_Analysis} will formalize.

\subsection{Distributed optimization algorithm}
The decoupling of the central NLP~\eqref{eq:central_NLP} into the local NLPs~\eqref{eq:local_NLP} is exploited by solving the individual problems in parallel at the agent level, see Algorithm~\ref{alg:SENSI_search_direction} which is denoted as SBDP+. This requires a bi-directional, neighbor-to-neighbor communication network for which we assume the same graph structure $\mathcal{G}$ as in the coupling structure of the central NLP~\eqref{eq:central_NLP}.
%%%%%%%%%%%%%%%%%%%%%%%%%%%%%%%%%%%%%%%%%%%%%%%%%%%%%%%%%%%%%%%
\begin{algorithm}[tb]\small
	\caption{SBDP+ for each agent $\agents$}
	\begin{algorithmic}[1]
		\setcounter{ALG@line}{-1}
		\State Initialize $\vm p_i^0$; Choose step sizes $\alpha, \,\beta >0$, penalty parameter $\rho_i\geq0$, and tolerance $\epsilon>0$; send $\vm x_i^0$ to $\neighs$; set $q\leftarrow 0$
		\State Compute $\nablaxj  L_i\indq$ via \eqref{eq:gradient_withoutneighboraffine} for all $\neighs$
		\State Send  $\nablaxj  L_i\indq$ to the corresponding neighbor $\neighs$
		\State Compute $ \vm y_i\indq$ by solving the decoupled NLP \eqref{eq:local_NLP} to local optimality
		\State Assemble $\vm P_i\indq(\vm y_i\indq)$ and compute $\vm p_i\indqn$ via the update rule \eqref{eq:primal_dual_update}
		\State Send $\vm x_i\indqn $ to all neighbors $\neighs$
		\State Stop if a suitable convergence criterion, e.g., $\|\vm s\indq \|_\infty \leq \epsilon$, is met. Otherwise, return to line~$1$ with $q \leftarrow q+1$.
	\end{algorithmic}\label{alg:SENSI_search_direction}
\end{algorithm}
%%%%%%%%%%%%%%%%%%%%%%%%%%%%%%%%%%%%%%%%%%%%%%%%%%%%%%%%%%%%%%%%%%
%
In Step 1, each agent computes the partial derivatives $\nablaxj L_i\indq$ which are shared with neighboring agents in Step 2 to enable the evaluation of the local cost function \eqref{eq:definition_costFunction}. Subsequently, each agent independently and in parallel solves the local NLP~\eqref{eq:local_NLP} in Step~3, after which the primal and dual variables are updated via \eqref{eq:primal_dual_update}. Then, the new decision variable $\vm x_i\indqn$ is communicated to the neighbors in Step~5. A practical tuning guideline for the step sizes $\alpha$ and $ \beta$ as well as the penalty parameter $\rho_i$ will be discussed in Section \ref{subsec:practicalTuning}.
A possible stopping criterion is 
$
	\|\vm s\indq\|_\infty \leq \epsilon
$
with tolerance $\epsilon>0$. However, other stopping criteria such as first order optimality or a fixed number of iterations are possible. Algorithm \ref{alg:SENSI_search_direction} only involves local computations and requires two neighbor-to-neighbor communication steps in which a maximum number of $\sum_{\agents}2n_i|\mathcal{N}_i|$ floats needs to be sent system-wide per SBDP+ iteration. The framework assumes access to a reliable low-level solver for the local NLPs \eqref{eq:local_NLP} as we suppose these problems are solved to (local) optimality, although in practice the solution accuracy may affect convergence properties.
\section{Algorithmic Analysis}
\label{sec:Algorithmic_Analysis}
The analysis of Algorithm \ref{alg:SENSI_search_direction} is structured into three parts. The first part presents the central and local Karush-Kuhn-Tucker (KKT) conditions along with necessary regularity assumptions. The second section investigates the convergence towards a central KKT point and contains the proof of convergence of Algorithm \ref{alg:SENSI_search_direction} by viewing SBDP+ as a nonlinear discrete-time system. The last section contains a practical tuning guideline for the involved algorithmic parameters. 

\subsection{Preliminaries and optimality conditions}
We characterize the fixed point of the algorithm, by establishing the KKT conditions of the central NLP~\eqref{eq:central_NLP} as
	\begin{subequations}\label{eq:centralKKT}
	\begin{alignat}{2}
		\vm 0 &= \nabla_{\vm x_i} L(\vm x, \vm \lambda, \vm \mu)  \label{eq:centralKKT_gradient}\,, \quad &&\agents \\
		\vm 0 &= \vm g_i(\vm x_i, \Ni{x})\,, \quad &&\agents\\
		\vm 0 & = \vm U_i \vm h_i(\vm x_i, \Ni{x}),\, \vm h_i(\vm x_i, \Ni{x})\leq\vm  0,\,  \vm \mu_i \geq \vm 0,\ &&\agents
	\end{alignat}
\end{subequations}
	where $\vm U_i:= \diag([\mu_{1,i},\dots,\mu_{n_{hi},i}])$ is a diagonal matrix consisting of each $\mu_{k,i}$ in $\vm \mu_i$, $k \in\mathbb{N}_{[1,n_{hi}]}$.
	Furthermore, let $\vm p\inds:= [{\vm x\inds}\trans,\, {\vm \lambda\inds}\trans,\, {\vm \mu\inds}\trans ]\trans$ be a KKT point of~\eqref{eq:central_NLP}. The set of active inequality constraints at point $\vm x$ is denoted as
	\begin{align}
		\mathcal{A}(\vm x) := \{k \in\mathbb{N}_{[1,n_h]}\,|\, [\vm h(\vm x)]_k = 0  \} \,,   \label{eq:activesets}
	\end{align}
	while the set of inactive inequalities is $\mathcal{I}(\vm x) := \mathbb{N}_{[1,n_h]}\setminus \mathcal{A}(\vm x)$. 
	We make the following regularity assumption. 
	\begin{assumption}\label{ass:constraint_regularity}
	There exists a KKT point of \eqref{eq:central_NLP} satisfying
	\begin{enumerate}[i)]
		\item  strict complementarity (SC), i.e., $ [\vm \mu\inds]_k> 0$, $k \in \mathcal{A}(\vm x\inds)$,
		\item the linear constraint qualification (LICQ), i.e., the matrix $[\nabla_{\vm x}\trans \vm g(\vm x\inds), [\nabla_{\vm x} \vm h(\vm x\inds)]_{\mathcal{A}(\vm x\inds)}\trans]$ has full column rank.
	\end{enumerate}
\end{assumption}
The next assumption strengthens the standard second-order sufficiency condition (SOSC)~\cite{Nocedal} by requiring the Hessian of the central Lagrangian \eqref{eq:central_Lagrangian} to be uniformly positive definite in all directions, not just those in the critical cone.
\begin{assumption} \label{ass:uniform_SOSC}
	The KKT point of \eqref{eq:central_NLP} satisfies a uniform SOSC (USOSC), i.e., $\nabla_{\vm x \vm x}^2L(\vm x \inds\!, \vm \lambda\inds\!, \vm \mu\inds) \!\succ\!0$.
\end{assumption}
A relaxation to the standard SOSC with the corresponding algorithmic modifications is discussed in Section \ref{sec:Alg_extensions}. Regarding the local NLPs \eqref{eq:local_NLP}, the KKT conditions are  
	\begin{subequations}\label{eq:localKKT}
	\begin{alignat}{2}
		\vm 0 &= \nabla_{\vm s_i} \bar f_i\indq(\vm s_i) +  \nabla_{\vm s_i}\trans \vm{\bar g}_i\indq(\vm s_i) 	\vm \nu_i+ \nabla_{\vm s_i}\trans \vm{\bar h}_i\indq(\vm s_i) \vm \kappa_i
		 \label{eq:localKKT_gradient} \\
		\vm 0 &= \vm{\bar g}_i\indq(\vm s_i)\\
		\vm 0 & = \vm K_i \vm{\bar h}_i\indq(\vm s_i), \quad \vm{\bar h}_i\indq(\vm s_i)\leq\vm  0,\quad  \vm \kappa_i \geq \vm 0
	\end{alignat}
\end{subequations}
for every $\agents$ in some iteration $q$. Similar to \eqref{eq:activesets}, the sets
	\begin{align}
		\mathcal{A}_i\indq(\vm s_i):= \{k \in \mathbb{N}_{[1,n_{hi}]}\, |\, [\vm{\bar h}_i\indq(\vm s_i)]_k = 0  \}\label{eq:activelocalsets}
	\end{align}
and $\mathcal{I}_i^q(\vm s_i) := \mathbb{N}_{1,n_{hi}}\setminus \mathcal{A}_i\indq(\vm s_i)$ denote the active and inactive inequality constraints of the modified NLPs~\eqref{eq:local_NLP}, respectively. We define the mapping $ \vm \Phi: \mathbb{R}^p \rightarrow \mathbb{R}^p$ as
\begin{equation}
\vm \Phi(\vm p) :=[\vm s(\vm p)\trans,\vm \nu(\vm p)\trans, \vm \kappa(\vm p)\trans]\trans\,,
\end{equation}
which denotes the stacked primal-dual solution of all KKT systems \eqref{eq:localKKT}. Note that $\vm \Phi(\vm p)$ is defined w.r.t. $\vm p$, since the local NLPs implicitly depend on $\vm p\indq$ through the functions $\bar f_i\indq (\cdot)$, $\vm{\bar g}_i\indq (\cdot)$, and $\vm{\bar h}_i\indq (\cdot)$, cf. \eqref{eq:definition_costFunction} to \eqref{eq:gradient_withoutneighboraffine}. In general, the mapping $\vm \Phi(\cdot)$ might neither exist, be expressible in closed form nor be single-valued. However, under certain regularity conditions $\vm \Phi(\cdot)$ is locally single-valued, i.e., $\vm y\indq = \vm \Phi(\vm p\indq)$ with $\vm y\indq = [\vm y_i\indq]_{\agents}$, and is continuously differentiable near $\vm p\inds$. 
These assumptions are specified in the following and can be enforced algorithmically. The first concerns the penalty parameter $\rho_i$ and the regularity of the local NLPs \eqref{eq:local_NLP} w.r.t. $\vm p\inds$.
\begin{assumption} \label{ass:localSOSC}
Let $\rho_i\geq 0$ be such that
\begin{align} \label{eq:localSOSC}
		\nablaxx L_i(\vm x_i\inds, \vm \lambda_i\inds, \vm \mu_i\inds, \Ni{x}\inds) + \rho_i 	\vm I \succ\vm 0\,, \quad \forall  \agents\,.
\end{align}
\end{assumption}

	Assumption \ref{ass:localSOSC} requires that $\rho_i$ must be chosen large enough such that the Hessian of the local Lagrangian functions \eqref{eq:local_NLP_Lagrangian} of the modified NLPs \eqref{eq:local_NLP} evaluated at the central KKT point $\vm p\inds$ are positive definite. If the Hessian of \eqref{eq:local_Lagrangian} is not positive definite, then \eqref{eq:localSOSC} can always be satisfied if $\rho_i>0$ is chosen larger than $|\underline{\lambda}(\nablaxx L_i(\vm x_i\inds, \vm \lambda_i\inds, \vm \mu_i\inds, \Ni{x}\inds))|$. Another assumption regards the graph-induced decomposition of the constraints which in this paper is assumed to be given. 
\begin{assumption}\label{ass:localLICQ}
We assume a constraint compatible decomposition (CCD) of the graph-structured NLP~\eqref{eq:central_NLP}, i.e., the matrix $[\nablax\trans \vm g_i(\vm x_i\inds,\vm x_{\mathcal{N}_i}\inds), [\nablax \vm h_i(\vm x_i\inds,\vm x_{\mathcal{N}_i}\inds)]_{\mathcal{A}_i\inds(\vm 0)}\trans]$ has full column rank for each agent $\agents$.
	\end{assumption}

	Assumption~\ref{ass:localLICQ} is a structural property of the graph-induced partitioning. It requires that the assignment of constraints to agents preserves linear independence of the active central constraint gradients. Hence, together with ii) in Assumption~\ref{ass:constraint_regularity}, it constitutes a decomposition-compatible strengthening of the centralized LICQ.
\begin{remark}
If Assumption \ref{ass:localLICQ} is violated, additional slack variables $\sigma_i \in \mathbb{R}^{n_{gi} + n_{hi}}$ can be introduced for each constraint in \eqref{eq:definition_constraints}, with positivity constraints for slack variables associated with inequality constraints. Then, Assumption \ref{ass:localLICQ} holds by construction since the local constraint Jacobians contain identity blocks. To ensure equivalence with NLP \eqref{eq:central_NLP}, the slack variables are penalized in \eqref{eq:definition_costFunction} as $\bar f_i(\vm s_i) + \rho^\sigma_i \|\vm \sigma_i\|_1 $. Under Assumption \ref{ass:constraint_regularity}, the slack vanishes at $\vm p\inds$ if $\rho_i^\sigma\geq \max\{\| \vm \lambda_i\inds\|_\infty, \|\vm \mu_i\inds\|_{\infty}\}$ \cite[Thm. 17.3]{Nocedal}. The non-smooth penalty can be transformed into a smooth formulation without affecting separability of the resulting local NLPs (cf. \cite[Sec. 17.2]{Nocedal}).
\end{remark}
For analysis purposes, we further consider the stacked local updates \eqref{eq:primal_dual_update} from a central viewpoint and define the operator $\vm T : \mathbb{R}^p \rightarrow \mathbb{R}^p$ associated with the SBDP+ recursion
\begin{align} \label{eq:SBDP_plus_operator}
		\vm T(\vm p) := \vm p + \alpha \vm P(\vm \Phi(\vm p), \vm p)(\vm \Phi(\vm p) - \vm d(\vm p))\,,
	 \end{align}
where the function $\vm P: \mathbb{R}^p \times \mathbb{R}^p \rightarrow \mathbb{R}^{p \times p}$ is given as
		\begin{equation} 
			\vm P(\vm y, \vm p\indq) \!=\! \begin{bmatrix}
		\nabla_{\vm s \vm s}^2 \!\sum_{\agents}\!\bar{L}_i^q(\vm y_i) 
			\!\!&\!\! \nabla_{\vm s}\trans\bar{\vm g}^q(\vm s) 
			\!\!&\!\! \nabla_{\vm s}\trans\bar{\vm h}^q(\vm s) \\
		-\beta \, \nabla_{\vm s} \bar{\vm g}^q(\vm s) 
			& \vm 0 & \vm 0 \\
		-\beta \, \vm K \, \nabla_{\vm s} \bar{\vm h}^q(\vm s) 
			& \vm 0 & -\beta \, \bar{\vm H}^q(\vm s)
\end{bmatrix}
		\end{equation}
with $\bar{\vm g}^q(\vm s) =[\vm {\bar g}_i\indq(\vm s_i)]_{\agents}$, $\bar{\vm h}^q(\vm s) =[\vm {\bar h}_i\indq(\vm s_i)]_{\agents}$, and block-diagonal matrices  $\vm {\bar H}(s) = \blkdiag(\vm {\bar H}_1\indq(\vm s_1),\dots, \vm {\bar H}_M\indq(\vm s_M))$ and $\vm K = \blkdiag(\vm K_1,\dots,\vm K_M)$. We abbreviate $\vm P\indq(\vm y) := \vm P(\vm y, \vm p\indq)$. The stacked offset vector is $\vm d(\vm p)= [\vm 0\trans, \vm \lambda\trans, \vm \mu\trans]\trans$.   Thus, Algorithm \ref{alg:SENSI_search_direction} can be written as fixed-point recursion
	\begin{align}\label{eq:primal_dual_update_stacked} 
			\vm p\indqn = \vm T(\vm p\indq)\,, \quad q = 0,1,\dots
	\end{align}                  
which is in contrast to SBDP, where $\vm T( \vm p) = \vm p + (\vm \Phi(\vm p) - \vm d(\vm p))$ and reveals the modification of the algorithmic structure by an iteration-dependent, nonlinear, primal-dual preconditioner.  	
\subsection{Local convergence}
The subsequent convergence analysis of Algorithm \ref{alg:SENSI_search_direction} will revolve around showing that the sequence $\{\vm p\indq\}$ generated by \eqref{eq:primal_dual_update_stacked} converges towards a (local) central solution $\vm p\inds$ of the KKT system \eqref{eq:centralKKT}. To associate the limit of this sequence with the KKT point of~\eqref{eq:central_NLP}, it is necessary to show that $\vm y\indq = \vm d(\vm p\indq) = [\vm 0\trans, {\vm \lambda\indq}\trans, {\vm \mu\indq}\trans]\trans$ implies that $\vm p\indq$ is also a KKT point of \eqref{eq:central_NLP}. This relation is established in the next lemma.
	\begin{lemma}\label{lem:optimality}
	Let Assumptions \ref{ass:constraint_regularity} to \ref{ass:localLICQ} hold. If $ \vm y\indq= \vm d(\vm p\indq)$, then $\vm p\indq$ is a locally unique KKT point of NLP \eqref{eq:central_NLP}. Conversely, if $\vm p \indq$ is a KKT point $\vm p\inds$, then $\vm \Phi(\vm p\inds) = \vm d(\vm p\inds)$ is a locally unique solution of \eqref{eq:localKKT} and $\vm p\inds = \vm T(\vm p\inds)$ holds locally. 
	\end{lemma}
	\begin{proof}
	See Appendix \ref{app:1}.
	\end{proof}

Lemma \ref{lem:optimality} is important from two perspectives. First, it establishes the optimality of the distributed solution $\vm p^{q}$ when the algorithm comes to a halt, that is when $\vm y\indq - \vm d(\vm p\indq) = \vm 0$ in \eqref{eq:SBDP_plus_operator}. In other words, the algorithm will make progress until $\vm p\indq$ satisfies the central optimality conditions. Second, it shows that $\vm \Phi(\vm p\inds) = \vm d(\vm p\inds)= [\vm 0\trans, {\vm \lambda\inds}\trans, {\vm \kappa\inds}\trans ]\trans$ is a local solution of the NLPs \eqref{eq:local_NLP} such that we can employ the basic sensitivity theorem \cite[Thm. 3.2.2]{Fiacco} to analyze the behavior of $\vm \Phi(\vm p)$ for $\vm p$ sufficiently close to $\vm p\inds$  which is adapted for the present case and summarized in Lemma  \ref{lem:solvability}. Furthermore, let $\Gamma_i^{\mathcal{A}/\mathcal{I}}(\cdot)$ be set-valued mappings that assign the elements of $\mathcal{A}_i^q(\vm s_i)$ or $\mathcal{I}_i^q(\vm s_i)$  their global index in $\mathbb{N}_{[1,n_{h}]} $, respectively.
\begin{lemma}\label{lem:solvability}
	Under Assumptions \ref{ass:constraint_regularity} to \ref{ass:localLICQ}, there exists a constant $r_1>0$ such that for $\vm p \in \mathcal{B}_{r_1}(\vm p\inds)$ it holds that
	\begin{enumerate}[i)]
		\item the mapping $\vm \Phi(\cdot)$ exists, is locally unique, and twice continuously differentiable with $\vm \Phi(\vm p)$ satisfying the local KKT conditions \eqref{eq:localKKT} for any $\vm p\indq = \vm p$,
		\item at any $\vm \Phi(\vm p)$ the sets $\mathcal{A}(\vm x\inds)$ and $\cup_{\agents}\Gamma_i^{\mathcal{A}}(\mathcal{A}_i(\vm s_i))$ of active inequalities of the central NLP \eqref{eq:central_NLP} and local NLPs~\eqref{eq:local_NLP} are equivalent, the LICQ for each NLP \eqref{eq:local_NLP} holds, SC is preserved, and we have $\bar{L}_i(\vm y) \succ \vm 0$.
	\end{enumerate}
\end{lemma}
\begin{proof}
	See Appendix \ref{app:2}.
\end{proof}

In essence, Lemma \ref{lem:solvability} ensures that the mapping $\vm \Phi(\cdot)$ is (locally) differentiable on $\mathcal{B}_{r_1}(\vm p\inds)$ and that the regularity assumptions regarding the central KKT point $\vm p\inds$ carry over to the KKT point $\vm \Phi(\vm p\indq)$ of the modified NLPs \eqref{eq:local_NLP} such that their solvability is ensured in iteration $q$, provided $\vm p\indq \in \mathcal{B}_{r_1}(\vm p\inds)$. Based on the differentiability of $\vm \Phi(\cdot)$, we can derive a first-order approximation of the error between current iterate $\vm p\indq$ and optimal solution $\vm p\inds$ of Algorithm \ref{alg:SENSI_search_direction}.
\begin{lemma}\label{lem:approximation}
	Suppose Assumptions  \ref{ass:constraint_regularity} to \ref{ass:localLICQ} hold. Then, a first-order approximation of the error $\Delta \vm p\indq := \vm p\indq - \vm p\inds$ of Algorithm~\ref{alg:SENSI_search_direction} in $ \mathcal{B}_{r_1}(\vm p\inds)$ at iteration $q$ is given by 
	\begin{align} \label{eq:linear_approx}
	\Delta \vm p\indqn = (\vm I - \alpha \vm A(\vm p\inds)) \Delta \vm p\indq + \vm r(\|\Delta \vm p\indq\|^2)\,,
	\end{align}
with $\vm I - \alpha \vm A(\vm p\inds)=\nabla \vm T(\vm p\inds)$. Hereby, $\vm r(\cdot) \in \mathcal{O}(\|\Delta \vm p\indq\|^2)$ with $\vm r(\vm 0) = \vm 0$ summarizes the higher-order terms and the matrix-valued function $\vm A: \mathbb{R}^p \rightarrow \mathbb{R}^{p\times p}$ is defined as
\begin{align}\label{eq:Dynamic_matrix}
	\vm A(\vm p) = \begin{bmatrix}\!
		\nabla_{\vm x \vm x}^2 L(\vm x, \vm \lambda,\vm \mu ) & \vm J_g(\vm x)\trans & \vm J_h(\vm x)\trans \\
		-\beta  \vm J_g(\vm x) & \vm 0 & \vm 0\\
		-\beta  \vm U\vm J_h(\vm x) & \vm 0 & -\beta\vm H(\vm x)
	\end{bmatrix}
\end{align}
with matrices $\vm J_g(\vm x)= \nabla_{\vm x} \vm g(\vm x)$, $\vm J_h(\vm x)= \nabla_{\vm x} \vm h(\vm x)$, $\vm H(\vm x)= \diag(\vm h(\vm x))$, and $\vm U = \blkdiag(\vm U_1,\dots,\vm U_M)$.
\end{lemma}
\begin{proof}
	See Appendix \ref{app:3}. 
\end{proof}
Although the matrix $\vm A(\vm p\inds)$ depends on the optimal solution and is not known beforehand, it can be used to analyze the local, asymptotic stability of the nonlinear iteration \eqref{eq:linear_approx} in the vicinity of $\vm p\inds$. This fact is used to establish local convergence of Algorithm~\ref{alg:SENSI_search_direction} as stated in the next theorem.

\begin{theorem} \label{th:conv}
	Let Assumptions \ref{ass:constraint_regularity} to \ref{ass:localLICQ} hold. Then, there exists an upper bound on the step size $\bar{\alpha}>0$ and a radius~$r>0$ such that if $\alpha < \bar{\alpha}$ and $\vm p^{0} \in \mathcal{B}_r(\vm p\inds)$ the sequence $\{\vm p\indq\}$ generated by Algorithm~\ref{alg:SENSI_search_direction} is bounded and converges asymptotically to a central KKT point $\vm p\inds$ of NLP~\eqref{eq:central_NLP}. Moreover, for any $\vm{Q}\succ \vm 0$ the discrete-time Lyapunov equation
	 \begin{equation}\label{eq:Lypunov_equation_theorem}
	(\vm I - \alpha \vm{ A})\trans \vm{\bar P} (\vm I - \alpha \vm{ A}) - \vm{\bar P} = - \vm {Q}
	\end{equation} 
	admits a positive-definite solution $\vm{\bar P}\succ \vm 0$ such that Algorithm~\ref{alg:SENSI_search_direction} converges Q-linearly w.r.t. the norm $\|\cdot\|_{\vm{\bar P}}$, i.e., for $q=1,2, \dots$ and $ C = \|\vm I - \alpha \vm A(\vm p\inds)\|_{\vm{\bar P}} \in (0,1)$ it holds that
	\begin{equation} \label{eq:linear_convergence_in_P_norm}
		\| \vm p\indq - \vm p\inds \|_{\vm{\bar P}} \leq C\|\vm p\indqp - \vm p\inds\|_{\vm{\bar P}}\,.
	\end{equation}
	Equivalently, the iterates converge R-linearly w.r.t.\ the norm $\|\cdot\|$, i.e., there exist constants $C_0>0$, $C_1\in(0,1)$ such that 
	\begin{align} \label{eq:R_linear_convergence}
	\| \vm p\indq - \vm p\inds \| \leq C_0C_1^q\|\vm p^0 - \vm p\inds\|\,.
	\end{align} 
\end{theorem}
\begin{proof}
	See Appendix \ref{app:4}. 
\end{proof}

Theorem \ref{th:conv} provides linear convergence of SBDP+, provided that the step size $\alpha$ in \eqref{eq:primal_dual_update} is sufficiently small and that the penalty $\rho_i$ in \eqref{eq:definition_costFunction} is chosen sufficiently large, see Assumption~\ref{ass:localSOSC}. Furthermore, it should be noted that convergence is ensured for any $\beta>0$ by an appropriate choice of the step size $\alpha$. 

In light of existing convergence results (cf. \cite[Thm. 5]{Scheu}, \cite[Thm. 1]{Schneider2}, \cite[Prop. 5]{Schneider3}, and \cite[Thm. 1]{Pierer}), the SBDP+ method proposed in Algorithm~\ref{alg:SENSI_search_direction} exhibits local convergence for a significantly broader class of NLPs, not just those satisfying the GDDC for NLP \eqref{eq:central_NLP}, i.e.,
$
	\|\vm I - \vm M(\vm p\inds)^{-1}\vm N(\vm p\inds)\| <1 
$
with $\vm M(\vm p\inds)$ and $\vm N(\vm p\inds)$ as defined in \eqref{eq:M} and \eqref{eq:N} in Appendix~\ref{app:3}. However, if the GDDC does hold, then SBDP alone, i.e., update \eqref{eq:SBDP_update}, guarantees linear (under certain conditions even quadratic) convergence w.r.t. the norm~${\|\cdot\|}$ \cite[Thm. 1]{Pierer}. Thus, from a practical perspective one can try SBDP first, and resort to SBDP+ if SBDP is ill-conditioned or fails to converge.
An interesting algorithmic simplification arises in the case of a central NLP \eqref{eq:central_NLP} that is not coupled via the constraints which is formalized in the following corollary.
 \begin{corollary} \label{cor:decoupled_constraints}
 Suppose that all the requirements of Theorem \ref{th:conv} hold and that the constraints are decoupled, i.e., $\vm g_i(\vm x_i, \Ni{x}) = \vm {g}_{ii}(\vm x_i)$ and $\vm h_i(\vm x_i, \Ni{x}) = \vm {h}_{ii}(\vm x_i)$, for all $\agents$. Then, we can choose $\vm P_i\indq(\vm y_i)=\vm I$, $\forall \vm y_i \in \mathbb{R}^{p_i}$ and Algorithm \ref{alg:SENSI_search_direction} retains the convergence properties of Theorem~\ref{th:conv}.
 \end{corollary}
\begin{proof}
	See Appendix \ref{app:6}. 
\end{proof}
This simplification is beneficial in two ways: it reduces computational complexity by avoiding the Hessian and Jacobian evaluations in~\eqref{eq:MixingMatrix}, and it shows that the transformation $\vm P_i^q(\vm y_i)$ in~\eqref{eq:primal_dual_update} is only required for coupled constraints.

\subsection{A practical parameter tuning guideline}
\label{subsec:practicalTuning}
Algorithm~\ref{alg:SENSI_search_direction} involves three tuning parameters: the primal step size $\alpha$, the dual step size $\beta$, and the penalty parameter $\rho_i$. Theorem~\ref{th:conv} and its assumptions provide implicit guidelines for selecting them to achieve good local performance. Since these conditions depend on the unknown KKT point $\vm p^\star$, we instead evaluate them at some $\vm p\in\mathcal{B}_r(\vm p^\star)$, e.g., the initial guess $\vm p^0$ or a small set of sampled points. This is justified by continuity of the involved functions on $\mathcal{B}_r(\vm p^\star)$. The resulting tuning procedure is summarized below.
\begin{enumerate}[1)]
	\item Choose the smallest possible penalty $\rho_i\geq 0$ such that $\nablaxx L_i(\vm x_i, \vm \lambda_i, \vm \mu_i, \Ni{x}) + \rho_i \vm I \succ\vm 0$ for all $\agents$.
	\item Perform Step 4 of Algorithm \ref{alg:SENSI_search_direction} and collect all local quantities. Choose the dual step size $\beta>0$ as 
	\begin{equation}  \label{eq:beta_update1}
		\beta(\vm p) = \frac{\underline \lambda(\nabla_{\vm x \vm x}^2 L(\vm {\bar x}, \vm \nu, \vm \kappa))}{\bar \lambda(\vm J(\bar{\vm x})\trans \vm{\bar K}\vm J(\bar{\vm x}))}
	\end{equation}
		with $ \vm {\bar x} = \vm x + \vm s$, $\vm J(\vm x) := [\nabla_{\vm x}\trans \vm g(\vm x), \nabla_{\vm x}\trans \vm h(\vm x)]\trans$, and $\vm {\bar K}=\diag([\mathds{1}\trans, \vm \kappa\trans])$ where $\mathds{1} \in \mathbb{R}^{n_g}$ is a unity vector.
	\item Choose $\alpha\in(0,1)$ such that $\max_i\{|1 - \alpha \lambda_i|\}<1$ for all eigenvalues $\lambda_i$, $i\in \mathbb{N}_{[0,p]}$ of $ \vm A(\vm p)$. 
\end{enumerate}
Although these steps do not offer convergence guarantees, they provide a principled initial estimate for the parameters, upon which further refinement can be based. Importantly, when the central NLP \eqref{eq:central_NLP} is an equality-constrained QP, these guidelines are independent of $\vm p$ and hold globally.

\begin{remark}
	The motivation for the choice of $\beta$ in 2) stems from the quadratic eigenvalue problem \eqref{eq:qaudratic_eigenvalue_problem} and aims at keeping the eigenvalues of its quadratic pencil reasonably well clustered. This improves the conditioning of $\vm A(\vm p)$, allowing for larger $\alpha$. The choice of $\alpha$ in 3) is motivated by enforcing the Schur stability of $\vm I - \alpha \vm A(\vm p)$, cf. \eqref{eq:linear_approx}. Although a distributed eigenvalue evaluation of the steps 2) and 3) is generally possible \cite{Penna}, and thus an adaptation of $\alpha$ and $\beta $ during the iterations, this aspect is not considered further.
\end{remark}

\section{Algorithmic Extensions Under Relaxed Second-Order Conditions}
\label{sec:Alg_extensions}
Up to this point, the formal convergence guarantee for Algorithm \ref{alg:SENSI_search_direction} depends on the uniform positive definiteness of the Hessian of the central Lagrangian~\eqref{eq:central_Lagrangian} at the central KKT point $\vm p\inds$, see Assumption \ref{ass:uniform_SOSC}. If this condition is not met, Algorithm \ref{alg:SENSI_search_direction} might not converge as shown in the next example.
\begin{example} \label{ex:2}
\textit{
Consider the optimization problem 
	\begin{alignat}{1} \label{eq:example_NLP}
		\min_{x_1,\, x_2} \quad x_1x_2 \quad 
	  \st \smallspace  x_1 - x_2 = 0 
	\end{alignat}
which is of the form \eqref{eq:central_NLP} with $f_{i}(x_i,x_j) = 0.5x_ix_j$, $i\in \{1,\,2\},\, j\neq i$ and $g_1(x_1,x_2)= x_1 -x_2$. Problem \eqref{eq:example_NLP} has a unique solution at $x_1\inds=x_2\inds=\lambda_1\inds = 0$ which satisfies the LICQ and SOSC but not $\nabla_{\vm x \vm x} L(\vm x\inds, \lambda_1\inds) \succ \vm 0$. Since \eqref{eq:example_NLP} is an equality constrained QP, the linear approximation in \eqref{eq:Dynamic_matrix} exactly represents the error progression of Algorithm \ref{alg:SENSI_search_direction}, where the eigenvalues of the matrix $\vm A(\vm p)$ in \eqref{eq:Dynamic_matrix} evaluate to $ \zeta_1 = 1$, $ \zeta_{2/3} = -0.5 \pm 0.5\sqrt{1- 8\beta}$. This demonstrates that the iteration \eqref{eq:linear_approx} diverges for any $\alpha$ or $\beta$, as $\mathrm{Re}(\zeta_{2/3})< 0$  due to the indefiniteness of $\nabla_{\vm x \vm x} L(\vm x\inds, \lambda_1\inds)$.}
\end{example}
To address this issue raised by Example \ref{ex:2}, we develop an algorithmic extension of Algorithm \ref{alg:SENSI_search_direction}, where we replace Assumption~\ref{ass:uniform_SOSC} with the standard SOSC \cite{Nocedal} and a partial SOSC which requires the Hessian of \eqref{eq:central_Lagrangian} to also be positive definite on the normal space of coupled constraints. 
The extension can be interpreted as a redesign of the primal-dual operator \eqref{eq:SBDP_plus_operator} in the sense that we add a penalty term that modifies the primal block of \eqref{eq:Dynamic_matrix} such that $\nabla \vm T(\vm p\inds)$ is Schur stable by an appropriate choice of $\alpha$ even when the original Hessian of \eqref{eq:central_Lagrangian} is indefinite off the tangent space of the constraints.

\subsection{SBDP+ under second-order sufficient conditions}
The following assumption formalizes the SOSC \cite{Nocedal}.
\begin{assumptionp}{2.a}\label{ass:SOSC}
	The KKT point $\vm p\inds$ of \eqref{eq:central_NLP} satisfies the SOSC, i.e., $\vm z\trans \nabla_{\vm x \vm x}^2L(\vm x \inds, \vm \lambda\inds, \vm \mu\inds) \vm z>0$ for all $\vm z \neq \vm 0 $ with $\nabla_{\vm x} \vm g(\vm x\inds) \vm z = \vm 0 $ and $[\nabla_{\vm x} \vm h(\vm x \inds )]_{\mathcal{A}(\vm x\inds)} \vm z= \vm 0$.
\end{assumptionp}
Clearly, Assumption \ref{ass:SOSC} is less restrictive than Assumption~\ref{ass:uniform_SOSC} since the Hessian of \eqref{eq:central_Lagrangian} only needs to be positive definite in admissible directions. 
To guarantee convergence of Algorithm \ref{alg:SENSI_search_direction} under the relaxed Assumption \ref{ass:SOSC}, we draw inspiration from augmented Lagrangian methods and compensate possible negative curvature in directions normal to the constraint manifold, cf. \cite[Sec. 3.2.1]{Bertsekas2}. Specifically, we modify the update \eqref{eq:primal_dual_update} to
\begin{align} \label{eq:primal_dual_update_v2}
	\vm p_i\indqn \!=\! \vm p_i \indq \!+\! \alpha \big[ \vm P_i\indq(\vm y_i\indq) (\vm y_i\indq \!-\! \vm d_i(\vm p_i\indq)) \!+ \!\!\!\!\!\!  \sum_{j \in {\mathcal{N}_i \cup \{i\}}} \!\! \!\!\!\!\vm P_{ij}\indq(\vm y_j\indq) \vm p_j\indq \big]
\end{align}
with the matrix-valued functions $\vm P_{ij}\indq : \mathbb{R}^{p_j} \rightarrow \mathbb{R}^{p_i \times p_j}$, defined as $\vm P_{ij}\indq(\vm y_j):=\gamma[\vm S_{ij}\indq(\vm y_j)\trans, \vm 0\trans, \vm 0\trans ]\trans$ with
\begin{align}
	\vm S_{ij}\indq(\vm y_j):= \nabla_{\vm x_i\indq}\trans \vm{\bar g}_j\indq\nabla_{\vm x_j\indq} \vm {\bar g}_j\indq +\nabla_{\vm x_i\indq}\trans \vm{\bar h}_j\indq \vm K_j^2  \nabla_{\vm x_j\indq} \vm {\bar h}_j\indq\label{eq:NeighborMixingMatrix}
\end{align}
for each $\agents$, $\neighs \cup \{i\}$. The matrices $\vm P_{ij}\indq(\vm y_j)$ only affect the primal update and integrate information of the primal search direction $\vm s_j$ of neighboring agents whereby $\gamma \in \mathbb{R}_{\geq0}$ describes an additional penalty parameter. However, similar to~\eqref{eq:gradient_withoutneighboraffine}, the functions $\vm P_{ij}\indq(\cdot)$ in \eqref{eq:NeighborMixingMatrix} may potentially involve decision variables of second-order neighbors of agent $i$. As with the sensitivities, this is addressed  by letting each agent $\agents$ compute the mirroring expression $(\nabla_{\vm s_j}\trans  \vm{\bar g}_i\indq(\vm s_i)\nabla_{\vm s_i} \vm{\bar g}_i\indq(\vm s_i) + \nabla_{\vm s_j}\trans \vm{\bar h}_i\indq(\vm s_i) \vm K_{i}^2 \nabla_{\vm s_i}  \vm{\bar h}_i\indq(\vm s_i)) \vm s_i = \vm S_{ji}\indq(\vm y_i)\vm s_i$ for each $\neighs$ instead and communicating it to the respective neighbors $\neighs$. Furthermore, the evaluation of the function $\vm S_{ji}\indq(\cdot) $ at $\vm y_i\indq$ only requires quantities that were already computed for the sensitivities \eqref{eq:gradient_withoutneighboraffine} and the update in \eqref{eq:primal_dual_update}. Thus, no additional computational effort is introduced. The process of solving the local NLPs \eqref{eq:local_NLP} remains unchanged.
 
The modified SBDP+ method under the SOSC assumption is given in Algorithm \ref{alg:SBDP_v2}. The basic algorithmic steps of Algorithm~\ref{alg:SENSI_search_direction} remain the same, we merely introduce the additional communication step 4 which requires the exchange of an extra $|\mathcal{N}_i| n_i$ floats per agent leading to a total communication load of $\sum_{\agents}3n_i|\mathcal{N}_i|$ floats per SBDP+ iteration and use the update rule \eqref{eq:primal_dual_update_v2} instead of \eqref{eq:primal_dual_update}. The communication requirement remains dependent on the size of the primal optimization vector and not the number of constraints, although it is increased w.r.t. the communication load of Algorithm~\ref{alg:SENSI_search_direction}. 
 \begin{algorithm}[tb]\small
 \caption{SBDP+ for each agent $\agents$ under SOSC}
 \begin{algorithmic}[1]
	\setcounter{ALG@line}{-1}
 	\State Initialize $\vm p_i^0$; Choose step sizes $\alpha, \beta>0$, penalty parameters $\rho_i,\gamma\geq0$, and tolerance $\epsilon>0$; send $\vm x_i^0$ to $\neighs$; set $q\leftarrow 0$
 	\State Compute $\nablaxj  L_i\indq$ via \eqref{eq:gradient_withoutneighboraffine} for all $\neighs$
 	\State Send  $\nablaxj  L_i\indq$ to the corresponding neighbor $\neighs$
 	\State Compute $ \vm y_i\indq$ by solving the decoupled NLP \eqref{eq:local_NLP} to local optimality
 	\State Send $\vm S_{ji}\indq(\vm y_i\indq) \vm s_i\indq$ to the corresponding neighbor $\neighs$
 	\State Assemble $\vm P_i\indq(\vm y_i\indq)$, $\vm P_{ii}\indq(\vm y_i\indq)$ and $\vm P_{ij}\indq(\vm y_j\indq)$, $\neighs$ and compute $\vm p_i\indqn$ via the update rule \eqref{eq:primal_dual_update_v2} or \eqref{eq:primal_dual_update_v3}
 	\State Send $\vm x_i\indqn $ to all neighbors $\neighs$
 	\State Stop if a suitable convergence criterion, e.g., $\|\vm s\indq \|_\infty \leq \epsilon$, is met. Otherwise, return to line~$1$ with $q \leftarrow q+1$.
 \end{algorithmic}\label{alg:SBDP_v2}
\end{algorithm}

\subsection{SBDP+ under partial second-order sufficient conditions}
We now consider the case of a partial SOSC (PSOSC) which causes a beneficial algorithmic simplification. To this end, we partition \eqref{eq:central_equality} -- \eqref{eq:central_inequality} into decoupled and coupled constraints
\begin{subequations} \label{eq:partitioning_into_local_and_coupled}
\begin{align}
	\vm g_i(\vm x_i, \Ni{x}) &= [\vm g_{ii}(\vm x_i)\trans, \vm g_{\mathcal{N}_i}(\vm x_i,\Ni{x})\trans ]\trans \\
	\vm h_i(\vm x_i, \Ni{x}) &= [\vm h_{ii}(\vm x_i)\trans, \vm h_{\mathcal{N}_i}(\vm x_i,\Ni{x})\trans ]\trans\,,
\end{align}
\end{subequations}
where $(\cdot)_{ii}$ depends exclusively on local variables and $(\cdot)_{\mathcal{N}_i}$ depends on neighboring and local variables. This leads to the reordered centralized constraints
\begin{equation}  \label{eq:reording_central_constraints}
	\vm g(\vm x) = [\vm{\bar{g}}(\vm x)\trans,\,\vm{\hat{g}}(\vm x)\trans]\trans\!,\,	\vm h(\vm x) = [\vm{\bar{h}}(\vm x)\trans,\,\vm{\hat{h}}(\vm x)\trans]\trans
\end{equation}
with $\vm{\bar{g}}(\vm x) := [\vm g_{ii}(\vm x_i)]_{\agents}$, $\vm{\hat{g}}(\vm x) := [\vm g_{\mathcal{N}_i}(\vm x_i,\Ni{x})]_{\agents}$,  $\vm{\bar{h}}(\vm x) := [\vm h_{ii}(\vm x_i)]_{\agents}$, and $\vm{\hat{h}}(\vm x) := [\vm h_{\mathcal{N}_i}(\vm x_i,\Ni{x})]_{\agents}$. 

In light of the partitioning \eqref{eq:partitioning_into_local_and_coupled} and the reordering \eqref{eq:reording_central_constraints}, we make the following assumption formalizing the partial SOSC. 
\begin{assumptionp}{2.b}\label{ass:regularity_local_constraints}
The KKT point $\vm p\inds$ of \eqref{eq:central_NLP} satisfies a partial SOSC (PSOSC), i.e., $\vm z\trans \nabla_{\vm x \vm x}^2L(\vm x \inds, \vm \lambda\inds, \vm \mu\inds) \vm z>0$ for all $\vm z \neq \vm 0 $ with $\nabla_{\vm x} \vm{\bar{g}}(\vm x\inds) \vm z = \vm 0 $ and $[\nabla_{\vm x}\vm{\bar{h}}(\vm x \inds )]_{\mathcal{A}(\vm x\inds)} \vm z= \vm 0$.
\end{assumptionp}

Assumption \ref{ass:regularity_local_constraints} is slightly stronger than the standard SOSC in Assumption \ref{ass:SOSC} since we exclude coupled constraints, i.e., the conditions $\nabla_{\vm x} \vm{\hat{g}}(\vm x\inds) \vm z = \vm 0 $ and $[\nabla_{\vm x} \vm{ \hat{h}}(\vm x \inds )]_{\mathcal{A}(\vm x\inds)} \vm z= \vm 0$, when demanding positive definiteness of the Hessian of \eqref{eq:central_Lagrangian} at the KKT point $\vm p\inds$.
This leads us to consider the following simplification of the update rule \eqref{eq:primal_dual_update_v2}
\begin{align} \label{eq:primal_dual_update_v3}
	\vm p_i\indqn = \vm p_i \indq + \alpha \big[ \vm P_i\indq(\vm y_i\indq) (\vm y_i\indq - \vm d_i(\vm p_i\indq)) + \vm P_{ii}\indq(\vm y_i\indq) \vm s_i\indq \big]\,,
\end{align}
where we only consider the correction term \eqref{eq:NeighborMixingMatrix} w.r.t. search direction $\vm s_i$, i.e., $\vm P_{ii}\indq(\vm y_i)=\gamma[\vm S_{ii}\indq(\vm y_i)\trans, \vm 0\trans, \vm 0\trans ]\trans$ with
\begin{align}
	\vm S_{ii}\indq(\vm y_i):= \nabla_{\vm x_i\indq}\trans  \vm{\bar g}_{ii}\indq \nabla_{\vm x_i\indq} \vm{\bar g}_{ii}\indq+\nabla_{\vm x_i\indq}\trans \vm{\bar h}_{ii}\indq\vm K_i^2  \nabla_{\vm x_i\indq} \vm{\bar h}_{ii}\indq \label{eq:AgentMixingMatrix}
\end{align} 
which only depends on locally available variables. 
This has the consequence that we can skip the communication Step 4 in Algorithm \ref{alg:SBDP_v2} and the computation of $\vm P_{ij}\indq(\vm y_j)$, $\neighs$, $j\neq i$ in Step 5 since $\vm P_{ii}\indq(\vm y_i)$ can be computed locally. This adaptation results in a change of the communication requirements of Algorithm \ref{alg:SBDP_v2} to two independent communication steps and a communication load of  $\sum_{\agents}2n_i|\mathcal{N}_i|$ floats per SBDP+ step which is identical to Algorithm \ref{alg:SENSI_search_direction}.

\subsection{Algorithmic analysis of Algorithm \ref{alg:SBDP_v2}}
The steps of the convergence analysis for Algorithm \ref{alg:SBDP_v2} stay largely the same as for Algorithm \ref{alg:SENSI_search_direction}. In particular, Lemma \ref{lem:optimality} and \ref{lem:solvability} remain valid since the local NLPs \eqref{eq:local_NLP} are unchanged. However, the update matrix $\vm P\indq(\vm y)$ in \eqref{eq:primal_dual_update_stacked} changes to
\begin{align}\label{eq:primal_dual_update_stacked_v2} 
\vm P\indq(\vm y) \!=\!\! \begin{bmatrix} 
	 \nabla_{\vm s \vm s}^2 \bar L\indq (\vm y) \!+\! \gamma \vm{\bar R}\indq(\vm y) \!\!&\!\!  \nabla_{\vm s}\trans \vm {\bar g}\indq(\vm s)  \!\!&\!\!  \nabla_{\vm s}\trans \vm {\bar h}\indq(\vm s)\\
	 -\beta \nabla_{\vm s} \vm {\bar g}\indq(\vm s) & \vm 0 & \vm 0\\
	-\beta \vm K \nabla_{\vm s} \vm {\bar h}\indq(\vm s) & \vm 0 & -\beta \vm H\indq(\vm s)
\end{bmatrix}
\end{align}
with ${\bar L}\indq (\vm y)= \sum_{\agents} {\bar L}_i\indq(\vm y_i)$, $\vm{\bar R}\indq(\vm y) = \nabla_{\vm s}\trans\vm{\bar g}\indq(\vm s) \nabla_{\vm s}\vm{\bar g}\indq(\vm s) + \nabla_{\vm s}\trans\vm{\bar h}\indq(\vm s) \vm K^2\nabla_{\vm s}\vm{\bar h}\indq(\vm s)$ and where additionally the complementarity condition $\vm K\indq \vm {\bar H}\indq(\vm s)$ is exploited. Considering the modified update matrix \eqref{eq:primal_dual_update_stacked_v2}, we can derive a first-order approximation of the error $\Delta \vm p\indq$ via the Jacobian of the modified operator $\vm T(\vm p)$ in \eqref{eq:SBDP_plus_operator}, similar to Lemma \ref{lem:approximation}. 
\begin{lemma}\label{lem:approximation_v2}
	Suppose Assumptions \ref{ass:constraint_regularity}, \ref{ass:SOSC}, \ref{ass:localSOSC}, and \ref{ass:localLICQ} hold. Then, a first-order approximation of Algorithm \ref{alg:SBDP_v2} under update rule~\eqref{eq:primal_dual_update_v2} is given by \eqref{eq:linear_approx} with 
\begin{align}\label{eq:Dynamic_matrix_v2}
	\vm A(\vm p) \!=\!\! \begin{bmatrix}\!
		\nabla_{\vm x \vm x}^2 L(\vm x, \vm \lambda,\vm \mu )\! +\! \gamma \vm R(\vm p)  \!\!&\!\! \vm J_g(\vm x) \!\!&\!\! \vm J_h(\vm x) \\
		\!-\beta  \vm J_g(\vm x) \!\!&\!\! \vm 0 \!\!&\!\! \vm 0\\
		\!-\beta  \vm U\vm J_h(\vm x) \!\!&\!\! \vm 0 \!\!&\!\! -\beta\vm H(\vm x)
	\end{bmatrix}
\end{align}
and $\vm R(\vm p) = \vm J_g(\vm x)\trans \vm J_g(\vm x) + \vm J_h(\vm x)\trans \vm U^2 \vm J_h(\vm x)$. 
\end{lemma}
\begin{proof}
	See Appendix \ref{app:7}. 
\end{proof}
Considering the new structure of the Jacobian $\nabla \vm T(\vm p\inds)$, we derive the following convergence result for Algorithm \ref{alg:SBDP_v2}.
\begin{theorem} \label{th:conv_v2}
	Let Assumptions \ref{ass:constraint_regularity}, \ref{ass:SOSC}, \ref{ass:localSOSC}, and \ref{ass:localLICQ} hold. Then, there exists an upper bound on the step size $\bar{\alpha}>0$, a lower bound on the penalty parameter $\bar \gamma<\infty$, and a radius~$r$ such that if $\alpha < \bar{\alpha}$, $\gamma > \bar{\gamma}$ and $\vm p^{0} \in \mathcal{B}_r(\vm p\inds)$ the sequence $\{\vm p\indq\}$ generated by Algorithm~\ref{alg:SBDP_v2} under the update rule \eqref{eq:primal_dual_update_v2} is bounded and converges asymptotically to a central KKT point $\vm p\inds$ of NLP~\eqref{eq:central_NLP}. Moreover, for all $q=1,2,\dots$ Algorithm \ref{alg:SBDP_v2} is linearly convergent in the sense of~\eqref{eq:linear_convergence_in_P_norm} or~\eqref{eq:R_linear_convergence}.
\end{theorem}
\begin{proof}
	See Appendix \ref{app:8}. 
\end{proof}
Theorem \ref{th:conv_v2} reveals that the penalty parameter $\gamma$ in \eqref{eq:primal_dual_update_v2} needs to be chosen sufficiently large to compensate the negative curvature of the Lagrangian in directions normal to the constraints in order to obtain the same convergence properties of Theorem~\ref{th:conv}. This fact is further illustrated with Example \ref{ex:2}.
\begin{contexample}
	\textit{
	Returning to the problem \eqref{eq:example_NLP}, we apply Algorithm \ref{alg:SBDP_v2} with update rule \eqref{eq:primal_dual_update_v2} and obtain the modified matrix $\vm A(\vm p)$ in \eqref{eq:Dynamic_matrix_v2} with the eigenvalues $ \zeta_1 = 1$ and $ \zeta_{2/3} = \gamma -0.5 \pm 0.5 \sqrt{4\gamma(\gamma -1) - 8 \beta +1}$ which have a positive real part for $\gamma> \bar \gamma = 0.5$, showing that Algorithm \ref{alg:SBDP_v2} is convergent for sufficiently small~$\alpha$. }
\end{contexample}
The following corollary regards the convergence of Algorithm~\ref{alg:SBDP_v2} with the update rule \eqref{eq:primal_dual_update_v3} under Assumption \ref{ass:regularity_local_constraints}. 
 \begin{corollary} \label{cor:stricter_SOSC}
	Let Assumptions \ref{ass:constraint_regularity}, \ref{ass:regularity_local_constraints}, \ref{ass:localSOSC}, and \ref{ass:localLICQ} hold. Then, Theorem \ref{th:conv_v2} remains valid for the iterates generated by Algorithm \ref{alg:SBDP_v2} under the update rule \eqref{eq:primal_dual_update_v3}.
\end{corollary}
\begin{proof}
	See Appendix \ref{app:9}. 
\end{proof}
Corollary \ref{cor:stricter_SOSC} shows that the additional communication step in Algorithm \ref{alg:SBDP_v2} is only required when the Hessian of \eqref{eq:central_Lagrangian} at $\vm p^\star$ is not positive definite on the normal space of the coupled constraints. Otherwise, local convergence is already guaranteed by the update law \eqref{eq:primal_dual_update_v3}. Moreover, Algorithm \ref{alg:SBDP_v2} retains the local convergence properties of Algorithm \ref{alg:SENSI_search_direction} when the tuning guideline in Section \ref{subsec:practicalTuning} is applied using the modified Hessian $\nabla_{\vm x \vm x}^2 L(\vm x,\vm \lambda,\vm \mu)+\gamma \vm R(\vm p)$ in \eqref{eq:beta_update1}, with sufficiently large $\gamma$ according to Theorem~\ref{th:conv_v2}.

\subsection{Summary of convergence results}
A summary of the convergence results of the respective SBDP variants is given in Table \ref{tab:conv_results}. The first column states the considered algorithm with corresponding update rule, the second the assumptions made on the KKT point of the central NLP, and the third denotes the order of convergence. Note that the centralized SC requirement always implies that SC holds for the local NLPs. Furthermore, the (U/P)SOSC assumptions are fundamentally different from the GDDC.  (U/P)SOSC is an intrinsic curvature condition on the centralized NLP \eqref{eq:central_NLP} and is independent of the graph-induced decomposition, whereas GDDC is a structural condition on the decomposition itself that regards the quality of the block-diagonal approximation. In particular, USOSC is invariant under repartitioning, while GDDC may hold or fail depending on the given decomposition. In general, neither assumption implies the other. 
\begin{table}[bt]
\caption{Summary of convergence results.}
\begin{tabular}{ 
  |c|c|c|c| }
 \hline
 algorithm & assumptions & conv. order \\
 \hline
Alg. \ref{alg:SENSI_search_direction} with \eqref{eq:primal_dual_update} & USOSC, SC, LICQ, CCD & linear (weight.) \\
Alg. \ref{alg:SBDP_v2} with \eqref{eq:primal_dual_update_v2} & SOSC, SC, LICQ, CCD& linear (weight.) \\
Alg. \ref{alg:SBDP_v2} with \eqref{eq:primal_dual_update_v3} & PSOSC, SC, LICQ, CCD & linear (weight.) \\
SBDP \cite{Pierer} & SOSC, SC, LICQ, CCD, GDDC & linear/quadratic \\
\hline
\end{tabular} \label{tab:conv_results}
\end{table}
\section{Numerical Validation and Evaluation}
\label{sec:Numerical_Evaluation}
Two numerical examples are presented to evaluate the performance of the SBDP+ method. The first involves a simple non-convex problem, used to validate Theorem \ref{th:conv} and compares SBDP+ to representative algorithms such as ADMM and ALADIN. The second example focuses on a classification problem to show the applicability to large-scale statistical learning problems. 
\subsection{Convergence properties}
We investigate the convergence properties of Algorithm~\ref{alg:SENSI_search_direction} for the following non-convex, constrained NLP \cite{Engelmann2}
	\begin{subequations}\label{eq:conv_prop_NLP}
	\begin{alignat}{1}
		\min_{x_1,\, x_2} \quad & 2(x_1 - 1)^2 + (x_2 - 2)^2  \label{eq:conv_prop_NLP_cost} 	\\
		\st \quad&  0 \geq -1 -x_1x_2 \label{eq:conv_prop_NLP_constraint1}\,,\smallspace 0 \geq-1.5 + x_1x_2  
	\end{alignat}
\end{subequations}
with decision variables $x_1, x_2 \in \mathbb{R}$. We allocate each variable to one of the agents $\agents =\{1,\,2\}$, define the local objectives $f_1(x_1,x_2) = 2(x_1 - 1)^2  $ and  $f_2(x_2,x_1) =(x_2 - 2)^2  $ together with the inequality constraints $h_1(x_1,x_2) = -1 -x_1x_2$ and $h_2(x_2,x_1) = -1.5 +x_1x_2$ yielding the structure~\eqref{eq:central_NLP}. Thus, the coupling in \eqref{eq:central_NLP} arises via constraints, while the cost function remains separable. The sensitivity terms in \eqref{eq:gradient_withoutneighboraffine} become $\nabla_{x_2} L_1(x_1,\mu_1,x_2) = -\mu_1 x_1 $ and $\nabla_{x_1} L_2(x_2,\mu_2,x_1) =\mu_2x_2 $. Subsequently, the local NLPs \eqref{eq:local_NLP} have the objectives $\bar f\indq_1(s_1) = 2(x_1\indq + s_1 -1)^2 + \rho_i s_1^2$ and $\bar f\indq_2(s_2) = (x_2\indq + s_2 -2)^2 + \rho_i s_2^2$ and local constraints  $h_1\indq(s_1)= -1 -(x_1\indq +s_1) x_2\indq$ and $h_2\indq(s_2)= -1.5 -x_1\indq (x_2\indq + s_2)$. 
We apply Algorithm~\ref{alg:SENSI_search_direction} and illustrate the applicability of Theorem~\ref{th:conv} by investigating the convergence behavior around the locally unique optimal solution of NLP \eqref{eq:conv_prop_NLP} at $\vm x\inds \approx [0.82,\,1.84]\trans$, $\vm \mu\inds \approx[0,\, 0.4]\trans $ satisfying Assumptions \ref{ass:constraint_regularity} to \ref{ass:localLICQ} for $\rho_i=0$. We set $\beta = \beta(\vm p\inds) \approx 2$ according to \eqref{eq:beta_update1}. At first, we compute the maximum step size $\bar \alpha $ by constructing $\vm A(\vm p\inds)$ as defined in \eqref{eq:Dynamic_matrix} with 
\begin{align}
\nabla_{\vm x \vm x}^2 L(\vm x\inds, \vm \mu\inds) \!=\!\! \begin{bmatrix}
4 \!&\! \mu_2\inds \\
\mu_2\inds \!&\! 2
\end{bmatrix},\nabla_{\vm x} \vm h(\vm x\inds) \!=\!\! \begin{bmatrix}
-x_2\inds \!&\! -x_1\inds\\
x_2\inds \!&\! x_1\inds
\end{bmatrix},
\end{align}
and $\vm H(\vm x\inds) = \diag([h_1(x_1\inds, x_2\inds),\, h_2(x_2\inds,x_1\inds)])$ and investigate for which $\alpha$ the matrix $\vm I - \alpha \vm A(\vm p\inds)$ is Schur stable. This leads to $\bar \alpha = 0.4$ such that we set $\alpha = 0.35$. Then, we solve the discrete-time Riccati equation \eqref{eq:Lypunov_equation_theorem} with $\vm Q= \vm I$ to obtain $\vm{ \bar P}$ and the convergence constant in \eqref{eq:linear_convergence_in_P_norm} as $C= 0.76$, and in \eqref{eq:R_linear_convergence} as $C_0 = 2.07$ and $C_1=0.88$. The computation of the convergence radius $r$ is more involved. For example, the radius $r_1$ in Lemma~\ref{lem:solvability} depends on continuity properties of the underlying functions and the size of the domain where the implicit function theorem applies \cite[Thm. 3.2.2]{Fiacco}, typically leading to conservative estimates. Thus, we estimate a practical region of attraction empirically by testing candidate initializations $\vm p^0$  and verifying that all iterates satisfy Statement ii) of Lemma~\ref{lem:solvability} and that the Lyapunov function $V(\Delta \vm p) = \Delta \vm p\trans \vm {\bar P} \Delta \vm p$ decreases monotonically. 
For the initial guess $\vm p_0 =[1.4,\, 1.4\,, 0\,,0]\trans$, the properties are satisfied and Figure \ref{fig:conv_curves} shows the progression of the error norms $\|\Delta \vm p\indq\|$ and $\|\Delta \vm p\indq\|_{\vm{\bar P}}$ as well as the convergence rate estimate $C^q\|\Delta \vm p^0\|_{\vm{\bar P}}$ of \eqref{eq:linear_convergence_in_P_norm}. We observe that $\|\Delta \vm p\indq\|$ is not monotonically decreasing and converges R-linearly in the sense of \eqref{eq:R_linear_convergence}, while the weighted norm $\|\Delta \vm p\indq\|_{\vm{\bar P}}$ decreases linearly, as stated by Theorem~\ref{th:conv}, until the residual norm plateaus around $10^{-8}$ due to the solution tolerance of the local NLPs~\eqref{eq:local_NLP}. For $\alpha = 0.6> \bar \alpha$, divergence is observed as the step size is too large. Furthermore, $\|\Delta \vm p\indq\|_{\vm{\bar P}}$ is always upper bounded by the (conservative) estimate $C^q\|\Delta \vm p^0\|_{\vm{\bar P}}$.
\begin{figure}[tb]
	\centering
	\includegraphics{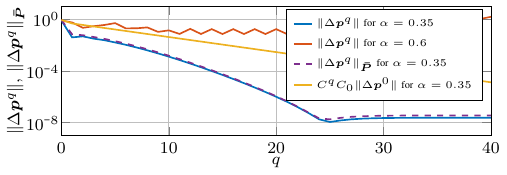}
	\vspace{-4mm}
	\caption{Errors $\|\Delta \vm p\indq\|$ and $\|\Delta \vm p\indq\|_{\vm P}$ for Algorithm \ref{alg:SENSI_search_direction} with envelope function $C^qC_0\|\Delta \vm p^0\|$ applied to NLP \eqref{eq:conv_prop_NLP} for $\vm p^0$ within $\mathcal{B}_r(\vm p\inds)$.}
	\label{fig:conv_curves}
\end{figure}
%%%%%%%%%%%%%%%%%%%%%%%%%%%%%%%%%%%%%%%%%%%%%%%%%%%%%%%%%%%%%%%%%%%%%%%%%%%%%%%%
\subsection{Comparison with SBDP, ADMM and ALADIN}
We compare the proposed SBDP+ method with ADMM~\cite{Boyd}, SBDP~\cite{Pierer}, ALADIN \cite{Houska} and the bi-level D-ALADIN variant \cite{Engelmann}, where the coordination QP is solved via ADMM. Further details on the algorithms are available in the respective references.
 \begin{figure}[tb]
 	\centering
	\includegraphics{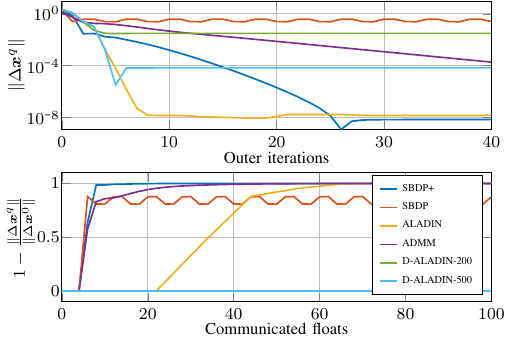}
 	\vspace{-4mm}
 	\caption{Comparison of SBDP+, SBDP, ALADIN, ADMM, and D-ALADIN with $200$ and $500$ inner ADMM iterations for NLP \eqref{eq:conv_prop_NLP} in terms of outer iterations (top) and communicated floats (bottom).}
 	\label{fig:Comparision_ADMM_ALADIN}
 \end{figure}
Figure \ref{fig:Comparision_ADMM_ALADIN} compares the convergence of SBDP(+), ALADIN, ADMM, D-ALADIN with $200$ and $500$ inner ADMM iterations applied to NLP \eqref{eq:conv_prop_NLP} in terms of outer iterations and communication effort, where for all algorithms the initial iterates are initialized to zero, e.g., $\vm p^0 = \vm 0$ for SBDP+. For ALADIN and D-ALADIN we use the implementation offered by the toolbox ALADIN-$\alpha$ \cite{Engelmann2} with the default parameter choice. The penalty parameter of ADMM is tuned to $1$ for good convergence in the set $\{0.1,\,1,\,10\}$. Clearly, ALADIN has the best performance since it offers local quadratic convergence guarantees. However, this performance comes at the price of centralized computations and increased communication effort which is visible in the lower plot of Figure \ref{fig:Comparision_ADMM_ALADIN}. While D-ALADIN distributes these computations, it only manages to converge to suboptimal solutions with a limited number of inner iterations \cite{Engelmann}. As in Example \ref{ex:1}, SBDP diverges as the KKT point of NLP \eqref{eq:conv_prop_NLP} violates the GDDC \cite{Pierer}, i.e., $\| \vm I - \vm M(\vm p\inds)^{-1} \vm N(\vm p\inds)\| \approx 5.5>1$.

\subsection{Application to distributed logistic regression}
As a second application example, we apply SBDP+ to the regularized logistic regression problem
\begin{subequations}\label{eq:log_regression}
	\begin{alignat}{1}
		\min_{\vm x} \quad & \frac{1}{m}\sum_{k=1}^{m} \log(1 + \exp(-b_k\vm a_k\trans \vm x)) + \frac{\epsilon}{2} \|\vm x\|^2  \label{eq:log_regression_cost} \\
		\st \quad&  \vm x_{\mathrm{min}} \leq \vm x \leq \vm x_{\mathrm{max}} \label{eq:log_regression_constraint} 
	\end{alignat}
\end{subequations}
with parameter vector $\vm x \in \mathbb{R}^n$, regularization parameter $\epsilon \in \mathbb{R}_{\geq0}$ and the training set which consists of $m$ pairs $(\vm a_k,b_k)$. Hereby, $\vm a_k \in \mathbb{R}^n$ is the feature vector and $b_k\in \{-1,1\}$ is the corresponding label. The box constraint \eqref{eq:log_regression_constraint} may incorporate prior knowledge. We randomly generate a problem instance with $m=200$ samples and $n=100$ features. The true weight vector $\vm x^{\mathrm{true}}$ is sampled from a standard normal distribution to generate the labels $b_k = \sign(\vm a_k\trans \vm x^{\mathrm{true}} + v_k)$, where $v_k \sim \mathcal{N}(0,0.1)$ is random noise. Following~\cite[Sec. 8.3]{Boyd}, the problem is decomposed across features into $M=10$ subsystems with variables $\vm x_i\in\mathbb{R}^{n/M}$. We set $\vm x_{\max}=-\vm x_{\min}=0.25\mathds{1}$ and apply Algorithm~\ref{alg:SENSI_search_direction} with $\rho_i=0.01$, $\alpha=0.85$, and $\vm P_i^q(\vm y_i)=\vm I$ according to Corollary~\ref{cor:decoupled_constraints} since NLP \eqref{eq:log_regression} is decoupled in the constraints.
Figure \ref{fig:conv_curves_logistic_regression} compares the error norm of SBDP+ and ADMM \cite[Sec. 8.3]{Boyd} by iteration to the central solution for $\epsilon = 0.1$. Both methods reach a numerical accuracy of around $10^{-6}$ limited by the local solver tolerance. ADMM uses a tuned penalty parameter of $0.1 \in \{0.01,\,0.1,\,1\} $, and requires solving a centralized QP, while SBDP+ relies solely on local computations.
 \begin{figure}[tb]
	\centering
	\includegraphics{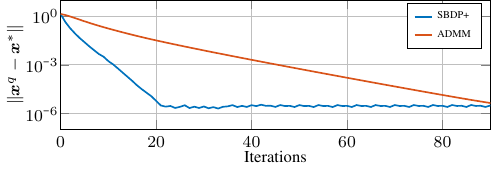}
	\vspace{-4mm}
	\caption{Error progression of Algorithm \ref{alg:SENSI_search_direction} with $\vm P_i\indq(\vm y_i) = \vm I$ applied to the regression problem \eqref{eq:log_regression} compared to ADMM with $M=10$.}
	\label{fig:conv_curves_logistic_regression}
\end{figure} 

\section{Conclusion}
\label{sec:Conclusion}
This paper introduces SBDP+, a framework for distributed optimization of non-convex, large-scale NLPs. It is characterized by decomposing the central NLP into small-scale NLPs at the agent level which are solved with neighbor-to-neighbor communication. Thus, SBDP+ pursues the approach of first decomposing and then optimizing. The coordination between agents takes place via sensitivities, which can be computed in a distributed fashion. The scheme is provably linearly convergent under standard regularity assumptions. Notably, SBDP+  ensures convergence for general coupling structures and avoids structural assumptions as required by SBDP.

Furthermore, we suggest an extension to deal with problems for which the Lagrangian is locally convex only on the constraints by incorporating constraint information of neighbors. 
The theoretical findings are validated in simulations, which also show that SBDP+ offers competitive performance compared to representative algorithms. Future work investigates the applicability of SBDP to convex optimization with non-differentiable objectives, different machine learning problems and explores the impact of inexact local NLP minimizations.
%%%%%%%%%%%%%%%%%%%%%%%%%%%%%%%%%%%%%%%%%%%%%%%%%%%%%%%%%%%%%%%%%%%%%%%%%%%%%%%%%%%%%%%%%%%%%%%%%%%%%%%%%%%%%%
%%%%%%%%%%%%%%%%%%%%%%%%% Appendix %%%%%%%%%%%%%%%%%%%%%%%%%%%%%%%%%%%%%%%%%%%%%%%%%%%%%%%%%%%%%%%%%%%%%%%%%%%
%%%%%%%%%%%%%%%%%%%%%%%%%%%%%%%%%%%%%%%%%%%%%%%%%%%%%%%%%%%%%%%%%%%%%%%%%%%%%%%%%%%%%%%%%%%%%%%%%%%%%%%%%%%%%%

\appendix
\subsection{Proof of Lemma \ref{lem:optimality}}
\label{app:1}
The first statement of the Lemma directly follows from the structural equivalence of the central KKT conditions~\eqref{eq:centralKKT} and the local KKT conditions \eqref{eq:localKKT} for  $\vm d(\vm p\indq) = \vm y\indq$, which implies $\vm s\indq = \vm 0$, $\vm \nu\indq = \vm \lambda\indq $ and $\vm \kappa\indq = \vm \mu\indq $. Thus, the iterate $\vm p \indq$ in \eqref{eq:primal_dual_update_stacked} satisfies \eqref{eq:centralKKT} and therefore also represents a KKT point $\vm p\inds$ of the central NLP~\eqref{eq:central_NLP}. Local uniqueness follows from Assumptions~\ref{ass:constraint_regularity} and \ref{ass:uniform_SOSC}, see \cite[Thm. 3.2.2]{Fiacco}. The second statement is obtained by inserting $\vm p\indq = \vm p\inds$ into \eqref{eq:localKKT} for every $\agents$ and observing that $\vm s = \vm 0 $, $\vm \nu = \vm \lambda\inds$ and $\vm \kappa = \vm \mu\inds$ is a (local) solution to this set of equations under the validity of \eqref{eq:centralKKT}. Furthermore, since the SOSC, LICQ and SC are satisfied at $\vm s = \vm 0$ due to Statement i) in Assumption~\ref{ass:constraint_regularity} as well as Assumptions \ref{ass:localSOSC} and \ref{ass:localLICQ}, the KKT point $\vm \Phi(\vm p\inds) = [\vm 0\trans, {\vm \lambda\inds}\trans, {\vm \kappa\inds}\trans ]\trans$ is also locally unique, see \cite[Thm. 3.2.2]{Fiacco}. By the definition of $\vm d(\vm p)$ in the update law \eqref{eq:primal_dual_update_stacked} and the operator definition \eqref{eq:SBDP_plus_operator} it immediately follows that  $\vm \Phi(\vm p\inds) = \vm d(\vm p\inds)$ and $\vm p\inds = \vm T(\vm p\inds)$ locally.\hfill \QEDclosed 

\subsection{Proof of Lemma \ref{lem:solvability}} 
\label{app:2}
Since  $\vm \Phi(\vm p\inds) = [\vm 0\trans, {\vm \lambda\inds}\trans, {\vm \kappa\inds}\trans ]\trans$ is a local solution of the NLPs \eqref{eq:local_NLP}, Assumptions \ref{ass:localSOSC} and \ref{ass:localLICQ} imply that the SOSC and the LICQ hold for each NLP~\eqref{eq:local_NLP}, $\agents$ at $\vm p\inds$. Strict complementarity for each NLP \eqref{eq:local_NLP} at $\vm p\inds$ for all inequality constraints follows from the equality of $\vm h(\vm x \inds) = [\vm h_i\inds(\vm 0)]_{\agents}$ and Statement i) in Assumption~\ref{ass:constraint_regularity}. Thus, with all involved functions being at least three times continuously differentiable, the conditions of the basic sensitivity-theorem \cite[Thm. 3.2.2]{Fiacco} and \cite[Cor. 3.2.5]{Fiacco} are satisfied which allows its application to each NLP~\eqref{eq:local_NLP} and proves the statements of the Lemma. \hfill \QEDclosed

\subsection{Proof of Lemma \ref{lem:approximation}}
\label{app:3}
To obtain a first-order approximation of Algorithm \ref{alg:SENSI_search_direction}, we consider the (nonlinear) stacked update \eqref{eq:primal_dual_update_stacked} and derive a linear approximation for deviations around the central KKT point $\vm p\inds$. With $\vm p\indq = \vm p\inds + \Delta \vm p\indq$, we have the linearization of~\eqref{eq:primal_dual_update_stacked} 
\begin{align} \label{eq:linearized_algorithm}
	\Delta \vm p\indqn \!\!=\!\! \big(\vm I \!+\! \alpha \vm P\inds(\vm \Phi(\vm p\inds))(\nabla_{\vm p}\vm \Phi(\vm p\inds) \!\! - \!\! \nabla_{\vm p}\vm d(\vm p\inds))\big)\Delta \vm p \indq
\end{align}
where $ \nabla_{\vm p} \vm P\inds(\vm \Phi(\vm p\inds))(\vm \Phi(\vm p\inds) - \vm d(\vm p\inds))= \vm 0$ from Lemma~\ref{lem:optimality} is used. The Jacobian of $\vm d(\vm p)$ at $\vm p\inds$ is computed in straightforward manner as $\nabla_{\vm p}\vm d(\vm p\inds) = \blkdiag(\vm 0, \vm I, \vm I)=: \vm D$. Furthermore, according to Lemma \ref{lem:solvability}, we differentiate $\vm \Phi(\vm p)$ w.r.t. $\vm p$ to yield an explicit representation of $\nabla_{\vm p}\vm \Phi(\vm p\inds)$. It follows that 
\begin{align} \label{eq:Jacobian_y}
	\nabla_{\vm p}\vm \Phi(\vm p\inds) = -\vm M(\vm p\inds)^{-1} (\vm N(\vm p\inds) - \vm M(\vm p\inds) \vm D)\,,
\end{align}
where the matrix function $\vm M: \mathbb{R}^p \rightarrow \mathbb{R}^{p\times p}$ is the Jacobian matrix of the equalities of the local KKT conditions \eqref{eq:localKKT} w.r.t. $\vm y$ evaluated at $\vm \Phi(\vm p)$ and some $\vm p\indq = \vm p$. It is computed as
 \begin{align}\label{eq:M}
	\vm M(\vm p) := \begin{bmatrix}  \vm{ \bar L}(\vm p) &  \vm{\bar J}_g(\vm x)\trans &  \vm{\bar J}_h(\vm x)\trans \\ \vm {\bar J}_g (\vm x)& \vm 0 & \vm 0 \\
		\vm K \vm {\bar J}_h(\vm x) & \vm 0 & \vm {\bar H}(\vm x)
	\end{bmatrix}
\end{align} 
with the Hessians $ \vm{ \bar L}(\vm p) = \blkdiag(\nabla_{\vm s_1 \vm s_1}^2 \bar{L}_1,\dots,\nabla_{\vm s_M \vm s_M}^2 \bar{L}_M)$, constraint Jacobians $ \vm {\bar J}_g(\vm x) = \blkdiag(\nabla_{\vm s_1} \vm{\bar g}_1,\dots,\nabla_{\vm s_M} \vm{\bar g}_M)$ and $ \vm {\bar J}_h(\vm x) = \blkdiag(\nabla_{\vm s_1} \vm{\bar h}_1,\dots,\nabla_{\vm s_M} \vm{\bar h}_M)$ as well as $\vm {\bar H }(\vm x) = \diag(\vm{\bar h}_i)$, evaluated at  $\vm \Phi(\vm p)$ and some $\vm p\indq = \vm p$. Furthermore, the Jacobian matrix of the equalities of the local KKT conditions \eqref{eq:localKKT} w.r.t. $\vm p\indq$, evaluated at $\vm \Phi(\vm p)$ and some $\vm p\indq = \vm p$, is $ (\vm N(\vm p) - \vm M(\vm p) \vm D)$, where the matrix function $\vm N :\mathbb{R}^p \rightarrow \mathbb{R}^{p\times p}$ is computed as
\begin{align}\label{eq:N}
	\vm N(\vm p) := \begin{bmatrix} \vm  L(\vm p) &    \vm J_g(\vm x)\trans &  \vm J_h(\vm x)\trans \\  \vm J_g(\vm x) & \vm 0 & \vm 0 \\
		\vm K \vm J_g(\vm x) & \vm 0 & \vm {\bar H}(\vm x)
	\end{bmatrix}
\end{align}
with the sub-matrices $ \vm{L}(\vm p) = \sum_{\agents} \nabla_{\vm x\indq \vm x\indq} \bar{L}_i\indq $, $ \vm J_g(\vm x) = \nabla_{\vm x\indq} [\vm{\bar g}_i\indq]_{\agents}$, $ \vm J_h(\vm x)=\nabla_{\vm x\indq} [\vm{\bar h}_i\indq]_{\agents}$, evaluated at $\vm \Phi(\vm p)$ and some $\vm p\indq = \vm p$.
The regularity of $\vm M(\vm p)$ for $\vm p \in \mathcal{B}_{r_1}(\vm p\inds)$ follows from i) and ii) in Lemma~\ref{lem:solvability} which is based on SC, cf. Statement~i) of Assumption \ref{ass:constraint_regularity}, the local solutions $\vm \Phi(\vm p\inds)$ of \eqref{eq:local_NLP} satisfying the SOSC and LICQ at $\vm p\inds$, cf. Assumptions \ref{ass:localSOSC} and~\ref{ass:localLICQ}, see \cite[Thm. 3.2.2]{Fiacco}, and \cite[Thm. 14]{Fiacco2}. Evaluating \eqref{eq:M} and \eqref{eq:N} at $\vm p\inds$ and inserting $\vm D$ and \eqref{eq:Jacobian_y} into \eqref{eq:linearized_algorithm}, gives
\begin{align} \label{eq:linearized_algorithm2}
\Delta \vm p\indqn = (\vm I - \alpha \vm P(\vm p\inds) \vm M(\vm p\inds)^{-1} \vm N(\vm p\inds))\Delta \vm p \indq\,.
\end{align}
Considering the definition of the update matrix \eqref{eq:MixingMatrix}, we can factorize $\vm P(\vm p \inds)$ as the matrix product
\begin{equation} \label{eq:preconditioning_matrix}
\vm P(\vm p\inds) = \blkdiag(\vm I, -\beta \vm I, - \beta \vm I)  \vm M(\vm p\inds)
\end{equation}
which, inserted into \eqref{eq:linearized_algorithm2} and explicitly including the higher-order terms, leads to the statement of the Lemma. The property $\vm r(\cdot) \in \mathcal{O}(\|\Delta \vm p\indq\|^2)$ follows from the twice differentiability of $\vm \Phi(\cdot)$ established in Lemma \ref{lem:solvability}, see \cite[Cor. 3.2.5]{Fiacco}.
\hfill \QEDclosed

\subsection{Proof of Theorem \ref{th:conv}}
\label{app:4}
We prove the theorem in three steps. First,~i), we show that there exists a sufficiently small step size $\bar \alpha$ such that the linear part of \eqref{eq:linear_approx}, i.e., $\Delta \vm p\indqn = \Delta \vm p\indq - \alpha \vm A(\vm p\inds) \Delta \vm p\indq$, is asymptotically stable. Second, ii), we prove the local convergence of $\{\vm p\indq\}$ to $\vm p\inds$ and estimate the radius of convergence~$r$ with discrete-time Lyapunov-based arguments. Finally, iii), we derive the rate estimates \eqref{eq:linear_convergence_in_P_norm} and \eqref{eq:R_linear_convergence}.

i) For simplicity, we proceed by deriving an alternative representation of the linear part of \eqref{eq:linear_approx} by reordering the active and non-active constraints. This is possible since the active sets of NLP \eqref{eq:central_NLP} and \eqref{eq:local_NLP} match within $\mathcal{B}_{r_1}(\vm p\inds)$, see Lemma~\ref{lem:solvability}. For convenience define the sets $\mathcal{A}\indq(\vm s):=\cup_{\agents}\Gamma_i^{\mathcal{A}}(\mathcal{A}_i\indq(\vm s_i))$ and $\mathcal{I}\indq(\vm s):=\cup_{\agents}\Gamma_i^{\mathcal{I}}(\mathcal{I}_i\indq(\vm s_i))$. By considering the non-active constraint information in the local KKT conditions \eqref{eq:localKKT}, we rewrite the linear part of~\eqref{eq:linear_approx} as
\begin{align} \label{eq:transformed_linear_approx}
	\begin{bmatrix}
	\Delta \vm{ \bar p}\indqn \\
	\Delta \vm{ \bar \mu}\indqn
	\end{bmatrix}
	 \!=\! \begin{bmatrix}
		\vm I  \!-\! \alpha \vm {\bar A} \!&\! \alpha[\nabla_{\vm x} \vm h(\vm x\inds)]_{\mathcal{I}(\vm x\inds)}\trans \\ \vm 0 \!&\!  \vm I \!+\! \alpha\beta [\vm H(\vm x\inds)]_{\mathcal{I}(\vm x\inds)}
	\end{bmatrix}\!	\begin{bmatrix}
	\Delta \vm{ \bar p}\indq \\
	\Delta \vm{ \bar \mu}\indq
\end{bmatrix}.
\end{align}
Hereby, we define the partitioning $\vm{ \bar p}: = [\vm x\trans, \vm \lambda\trans, [\vm \mu]_{\mathcal{A}\indq(\vm s\indq)}\trans]\trans$ and $\vm{ \bar \mu} := [\vm \mu]_{\mathcal{I}^q(\vm s\indq)} = [\vm \mu]_{\mathcal{I}(\vm x\inds)} $ since the current and optimal active set are identical for $\vm p\indq \in \mathcal{B}_{r_1}(\vm p\inds)$, i.e., $\mathcal{I}^q(\vm s\indq) = \mathcal{I}(\vm x\inds)$, see Lemma \ref{lem:solvability}. Moreover, we have the sub-matrix
\begin{align} \label{eq:partitioned_A_matrix}
	\vm {\bar A} = \begin{bmatrix}
	\vm L(\vm p\inds) & \vm J(\vm x\inds)\trans \\ - \beta \vm {\bar U}\inds \vm J(\vm x\inds) & \vm 0
\end{bmatrix}
\end{align}
with the Hessian $L(\vm p\inds) = \nabla_{\vm x \vm x} L(\vm x\inds, \vm \lambda\inds, \vm \mu\inds)$, $\vm J(\vm x\inds) := [\nabla_{\vm x}\trans \vm g(\vm x\inds), [\nabla_{\vm x} \vm h(\vm x\inds)]_{\mathcal{A}(\vm x\inds)}\trans]\trans$ as the Jacobian of active constraints, and $\vm {\bar U}\inds =\diag([\mathds{1}\trans, [\vm \mu\inds]_{\mathcal{A}(\vm x\inds)]}\trans
])\succ \vm 0$, where $\mathds{1} \in \mathbb{R}^{n_g}$ is the unity vector. Furthermore, we use $[\vm \mu\inds]_{\mathcal{I}(\vm x\inds)} = \vm 0$ as well as $[\vm h(\vm x\inds)]_{\mathcal{A}(\vm x\inds)} = \vm 0$ for the simplification in \eqref{eq:partitioned_A_matrix}. The system \eqref{eq:transformed_linear_approx} is asymptotically stable if and only if both sub-matrices, i.1), $ \vm I - \alpha \vm {\bar A}$ and, i.2), $ \vm I+ \alpha\beta [\vm H(\vm x\inds)]_{\mathcal{I}(\vm x\inds)}$ are Schur stable.

Regarding i.1), it is well known that the matrix $ \vm I - \alpha \vm {\bar A} $ is Schur stable if and only if $\mathrm{Re}(\lambda_i)>0$ and $\max_{i}\{ |1 - \alpha\lambda_i|\}<1$, where $\lambda_i \in \mathbb{C}$, $\forall i \in \mathbb{N}_{[1,p - |\mathcal{I}(\vm x\inds)|]}$, are the eigenvalues of $\vm{\bar A}$. We prove the first condition, i.e., $\mathrm{Re}(\lambda_i)>0$. To this end, we inspect the eigenvalue problem $ \vm{\bar A} \vm v = \lambda \vm v$, where we partition the eigenvector $\vm v = [\vm v_x\trans , \vm v_y\trans ]\trans $ according to the primal and dual variables in $\vm{\bar p}$ which leads to
\begin{subequations}
\begin{align}
\vm L(\vm p\inds) \vm v_x +  \vm J(\vm x\inds)\trans\vm v_y &= \lambda \vm v_x \label{eq:eigenvalue1} \\
-\beta \vm {\bar U}\inds \vm J(\vm x\inds) \vm v_x &= \lambda\vm v_y \label{eq:eigenvalue2}\,.
\end{align}
\end{subequations}
We first exclude the case of a zero eigenvalue which is shown by contradiction. Suppose that $\lambda =0$. Then, we obtain from \eqref{eq:eigenvalue1} that $\vm L(\vm p\inds) \vm v_x =- (\vm J\inds)\trans\vm v_y$ which multiplied with $\vm v_x\trans$ from the left results in
\begin{equation}\label{eq:contradiction_argument}
	\vm v_x\trans \vm L(\vm p\inds) \vm v_x = (- \vm J(\vm x\inds)\vm v_x)\trans \vm v_y\,.
\end{equation}
 However, from \eqref{eq:eigenvalue2} and $\beta > 0$ it follows that $\vm J(\vm x\inds)\vm v_x = \vm 0 $ and Assumption \ref{ass:uniform_SOSC} we know that $\vm v_x\trans \vm L(\vm p\inds) \vm v_x >  0$ which is a contradiction to \eqref{eq:contradiction_argument}. Thus, we proved that $ \lambda \neq 0 $. Consequently, we solve \eqref{eq:eigenvalue2} for $\vm v_y$ and insert the expression into \eqref{eq:eigenvalue1}, leading to the quadratic eigenvalue problem
 \begin{equation}\label{eq:qaudratic_eigenvalue_problem}
(\lambda^2 \vm I - \vm L(\vm p\inds) \lambda + \beta  (\vm J(\vm x\inds)\trans \vm {\bar U}\inds \vm J(\vm x\inds))\vm v_x = \vm 0\,.
 \end{equation}
Since the matrices $\vm I$ and $\vm L(\vm p\inds)$ are positive definite and $\vm J(\vm x\inds)\trans \beta \vm {\bar U}\inds \vm J(\vm x\inds)$ is positive semi-definite, it follows that $\mathrm{Re}(\lambda)\geq0$ \cite[Sec. 3.8]{Tisseuer}. However, this can be strengthened as we know that $\lambda \neq 0 $ and that no purely imaginary eigenvalues can exist due to $\vm L(\vm p\inds) \succ \vm 0$ which together imply that $\mathrm{Re}(\lambda)>0$. Note that this holds for any $\beta >0$.  We show that the second condition  $\max_{i}\{|1 - \alpha\lambda_i| \}<1$ is satisfied for sufficiently small $\alpha$. With $\lambda_i = \mathrm{Re}(\lambda_i) + j\mathrm{Im}(\lambda_i)$, we have 
\begin{align}\label{eq:cond2_on_eigenvalues}
|1 - \alpha \mathrm{Re}(\lambda_i) - \alpha j\mathrm{Im}(\lambda_i)|<1\,, \smallspace \forall i \in \mathbb{N}_{[1,p - |\mathcal{I}(\vm x\inds)|]}
\end{align}
which can be rewritten as a condition on $\alpha$
\begin{align}
\alpha < \min_{i \in \mathbb{N}_{[1,p - |\mathcal{I}(\vm x\inds)|]}} \bigg\{  \frac{2 \mathrm{Re}(\lambda_i)}{|\lambda_i|^2} \bigg\} =: \alpha_1\,.
\end{align}
This proves the Schur stability of $\vm I - \alpha \vm {\bar A}$ for $\alpha <\alpha_1$.

Regarding i.2), the Schur stability of $ \vm I+ \alpha\beta [\vm H(\vm x\inds)]_{\mathcal{I}(\vm x\inds)}$ is equivalent to the condition $ |1 + \alpha\beta[\vm h(\vm x\inds)]_k|< 1 $, $\forall k \in \mathcal{I}(\vm x\inds)$. Since it holds that $[\vm h(\vm x\inds)]_k<0$ due to SC and the inactivity of the constraint, we arrive at the condition 
\begin{equation}\label{eq:alpha_2}
\alpha < \frac{2}{\beta \max_{k \in \mathcal{I}(\vm x\inds)}\{|[\vm h(\vm x\inds)]_k|\}}=: \alpha_2
\end{equation}
such that for $\alpha < \alpha_2$, $\vm I+ \alpha\beta [\vm H(\vm x\inds)]_{\mathcal{I}(\vm x\inds)}$ is Schur stable. 
This proves the asymptotic stability of the iteration \eqref{eq:transformed_linear_approx}  for $\alpha < \min\{\alpha_1,\alpha_2\}:= \bar \alpha$.

ii) We now turn toward the nonlinear iteration \eqref{eq:linear_approx} that is $ \Delta \vm{ p}\indq = (\vm I - \alpha \vm { A}) \Delta \vm{ p}\indq + \vm { r}(\|\Delta \vm{ p}\indq\|^2)$. The property $\vm { r}(\cdot) \in \mathcal{O}(\|\Delta \vm{ p}\indq\|^2)$ implies the existence of constants $r_2,d>0$ such that $\|\vm { r}(\Delta \vm p)\| \leq d \|\Delta \vm p\|^2$ for  $\|\Delta \vm p\| \leq r_2$. As $\vm I - \alpha \vm{ A}$ is Schur stable for $\alpha <  \bar \alpha$, the discrete-time Lyapunov equation
\begin{equation}\label{eq:Lypunov_equation}
(\vm I - \alpha \vm{ A})\trans \vm{\bar P} (\vm I - \alpha \vm{ A}) - \vm{\bar P} = - \vm {Q}
\end{equation}
has a positive-definite solution $\vm{\bar P} \succ \vm 0 $ for some $\vm{Q} \succ \vm 0$. Define the Lyapunov function $V(\Delta \vm p) =\Delta \vm p\trans \vm{\bar P} \Delta \vm p $. Then, for $\|\Delta \vm p\| \leq \min\{r_1,r_2\}$, the change $\Delta V (\Delta \vm p\indq) = V (\Delta \vm p\indqn) \!-\! V (\Delta \vm p\indq)$ along the nonlinear iteration \eqref{eq:transformed_linear_approx} is bounded as 
\begin{align} \label{eq:Lypunov_equation_condition}
\Delta V (\Delta \vm p\indq) \leq  ( a\|\Delta \vm p\indq\|^2  +  b\|\Delta \vm p\indq\| - c )\|\Delta \vm{ p}\indq\|^2
\end{align}
with constants $a= \bar \lambda(\vm{\bar P}) d^2>0$, $b=d\|(\vm I - \alpha\vm { A})\trans\vm{\bar P}\|>0 $, and $c=\underline \lambda(\vm Q)>0$. Let $e \in \mathbb{R}_{>0}$ be a constant such that $e<c$ and suppose that
\begin{align} \label{eq:condition_deltap}
\|\Delta \vm{ p }\indq\| \leq \frac{-b + \sqrt{b^2 + 4a(c-e)}}{2a}=:r_3\,.
\end{align}
Insert \eqref{eq:condition_deltap} for $\|\Delta \vm{ p }\indq\|$ only within the bracketed term in \eqref{eq:Lypunov_equation_condition}, which results in $\Delta V (\Delta \vm p\indq)\leq -e \| \Delta \vm p\indq\|^2$. Thus, the point $\Delta \vm{ p} = \vm 0$ is (locally) exponentially stable \cite[Thm. 28 in Sec 5.9]{Vidyasagar}.
To characterize the radius of convergence, let $r_4 = \min\{r_1,r_2,r_3\}$ and compute $v_1 = \min_{\|\Delta \vm p\|=r_4} V(\Delta \vm p)>0$. Take $v_2 \in (0,v_1)$ and define the set $\Omega=\{\vm p \in \mathcal{B}_{r_4}(\vm p\inds)\,|\, V(\Delta \vm p)\leq v_2\}$. Since $V(\vm 0)=0$ and $V(\cdot)$ is continuous, there exists some $r<r_4$ such that $\vm p \in \mathcal{B}_{r}(\vm p\inds)$ implies $\vm p \in \Omega \subset \mathcal{B}_{r_4}(\vm p\inds) $. Due to the Lyapunov decrease condition $\Delta V (\Delta \vm p\indq)\leq -e \| \Delta \vm p\indq\|^2$, we have $\vm p^q \in \Omega \subset \mathcal{B}_{r_4}(\vm p\inds)$ for all $q=1,2,\dots$ if $\vm p^0 \in \mathcal{B}_r(\vm p\inds)$.
 Thus, for $\vm p^0 \in \mathcal{B}_r(\vm p\inds)$ the iterates defined by Algorithm~\ref{alg:SENSI_search_direction} are bounded and have the limit point $\lim_{q \rightarrow \infty} \vm{ p}\indq = \vm p\inds$ \cite[Sec. 5.9]{Vidyasagar}. 

iii) To derive the Q-linear convergence property, let $\vm Q =(1-\delta) \vm{\bar P}$ with $0<\delta<1$, then according to the Lyapunov equation~\eqref{eq:Lypunov_equation}, we have $(\vm I - \alpha \vm{ A})\trans \vm{\bar P} (\vm I - \alpha \vm{ A}) = (1 - \delta)\vm{\bar P}$ which implies that $\|\vm I - \alpha \vm A(\vm p\inds)\|_{\vm{\bar P}}= \sqrt{(1 - \delta)}<1$. This in turn shows $\|\Delta \vm p\indq\|_{\vm{\bar P}}\leq \|\vm I - \alpha \vm A(\vm p\inds)\|_{\vm{\bar P}}\|\Delta \vm p\indqp\|_{\vm{\bar P}} =C \|\Delta \vm p\indqp\|_{\vm{\bar P}} $, proving \eqref{eq:linear_convergence_in_P_norm} at iteration $q\geq 1$ with $C:=\|\vm I - \alpha \vm A(\vm p\inds)\|_{\vm{\bar P}}$. To derive the R-linear convergence property, consider that at each iteration $q=1,2, \dots$ it holds that 
\begin{align}\label{eq:Lyapunov_decrease}
&V(\Delta \vm { p}\indq)\leq V(\Delta \vm { p}\indqp)  -e \| \Delta \vm { p}\indqp\|^2 \leq p V(\Delta \vm { p}\indqp)  
\end{align}  
with the contraction factor $p := (1- e/\bar \lambda(\vm{\bar P}) )$. Recursively applying \eqref{eq:Lyapunov_decrease} and $\underline \lambda(\vm{\bar P})\|\Delta\vm {p}\|^2 \leq V(\Delta \vm{ p})\leq  \bar \lambda (\vm{\bar P}) \|\Delta \vm {   p}\|^2$ we arrive at $
\| \Delta \vm p\indq \| \leq C_0 C_1^{q}\|\Delta \vm p^0\|$
with $C_0 = \sqrt{\bar \lambda(\vm{\bar P}) /\underline \lambda(\vm{\bar P})}$ and $C_1 = \sqrt{p}$ which implies R-linear convergence \cite{Nocedal}.

\subsection{Proof of Corollary \ref{cor:decoupled_constraints}}
\label{app:6}
Setting $\vm P(\vm p) = \vm I$ and considering \eqref{eq:linearized_algorithm2} and \eqref{eq:preconditioning_matrix}, the linear part of the first-order approximation \eqref{eq:linear_approx} changes to 
$	\Delta \vm p\indqn = (\vm I - \vm M(\vm p\inds)^{-1} (\vm N(\vm p\inds) - \vm M(\vm p\inds) \vm D) - \vm D) \Delta \vm p\indq $. With $[\vm \kappa\inds]_{\mathcal{I}(\vm x\inds)}= \vm 0$, we obtain a partitioning similar to \eqref{eq:transformed_linear_approx}
 \begin{align} \label{eq:transformed_linear_approx_2}
 	\begin{bmatrix}
 		\Delta \vm{ \bar p}\indqn \\
 		\Delta \vm{ \bar \mu}\indqn
 	\end{bmatrix}
 	\!=\! \begin{bmatrix}
 		\vm I  \!-\! \alpha \vm {\bar A} & \vm 0 \\ \vm 0 \!&\!  (1-\alpha) \vm I
 	\end{bmatrix}\!	\begin{bmatrix}
 		\Delta \vm{ \bar p}\indq \\
 		\Delta \vm{ \bar \mu}\indq
 	\end{bmatrix}
 \end{align}
 with the sub-matrix
\begin{align}\label{eq:partitioned_A_matrix_corr1}
\vm{\bar{A}} =\begin{bmatrix}
	\vm{ \bar L}(\vm p\inds)& \!\!\vm J(\vm x\inds)\trans \\ \vm J(\vm x\inds)& \vm 0
\end{bmatrix}^{-1} \begin{bmatrix}
	\vm L(\vm p\inds) & \vm J(\vm x\inds)\trans \\ \vm J(\vm x\inds) & \vm 0
\end{bmatrix}
\end{align}
which can be viewed as preconditioning the KKT matrix of NLP \eqref{eq:central_NLP} with the indefinite matrix $\vm M(\vm p\inds$) such that we apply Theorem 2.1 in \cite{Keller} which states that $\vm{\bar{A}}$ has an eigenvalue at $1$ with multiplicity $2(n_g + |\mathcal{A}(\vm x\inds)|)$ and $n - (n_g + |\mathcal{A}(\vm x\inds)|)$ eigenvalues defined by the generalized eigenvalue problem $\vm Z\trans \vm{ \bar L}(\vm p\inds) \vm Z \vm v = \lambda \vm Z\trans \vm L(\vm p\inds) \vm Z $, where $\vm Z$ is a basis of the nullspace of $\vm J(\vm x\inds)$. Since both $\vm{ \bar L}(\vm p\inds)$ and $\vm{L}(\vm p\inds)$ are positive definite by Assumption \ref{ass:uniform_SOSC} and \ref{ass:localSOSC}, respectively, and $\vm Z$ has full column rank by definition, this generalized eigenvalue problem has only positive real eigenvalues.
Thus, the proof of convergence follows the proof of Theorem \ref{th:conv} from \eqref{eq:cond2_on_eigenvalues} onward with the modified dynamic matrix \eqref{eq:partitioned_A_matrix_corr1} and the fact that the iteration $	\Delta \vm{ \bar \mu}\indqn 
= (1- \alpha) \vm I \Delta \vm{ \bar \mu}\indq$ is trivially stable for $\alpha \in (0,1]$ such that we set $\alpha_2=1$ in \eqref{eq:alpha_2}. \hfill \QEDclosed

\subsection{Proof of Lemma \ref{lem:approximation_v2}}
\label{app:7}
Similar to \eqref{eq:preconditioning_matrix} in the proof of Lemma \ref{lem:approximation} and considering the modified update matrix \eqref{eq:primal_dual_update_stacked_v2}, we can factorize $\vm P(\vm p\inds)$ as
\begin{equation} \label{eq:preconditioning_matrix_v2}
	\vm P(\vm p\inds) = \begin{bmatrix}
		\vm I & \gamma \vm J_g(\vm x\inds)\trans & \gamma \vm J_h(\vm x\inds)\trans \vm U\inds  \\
		\vm 0 & -\beta \vm I & \vm 0 \\
		\vm 0 & \vm 0 &  -\beta \vm I
	\end{bmatrix} \vm M(\vm p\inds)
\end{equation}
and insert $\vm P(\vm p\inds)$ into \eqref{eq:linearized_algorithm2} which leads to \eqref{eq:Dynamic_matrix_v2}. \hfill \QEDclosed

\subsection{Proof of Theorem \ref{th:conv_v2}}
\label{app:8}
The steps i) to iii) in the proof of Theorem \ref{th:conv} are also applied here. However, due the modified linear approximation \eqref{eq:Dynamic_matrix_v2}, we consider the following partitioned system instead of \eqref{eq:partitioned_A_matrix} with the modified sub-matrix 
\begin{align} \label{eq:partitioned_A_matrix_v2}
	\vm {\bar A} = \begin{bmatrix}
		\vm L(\vm p\inds) + \gamma \vm R(\vm p\inds) & \vm J(\vm x\inds)\trans \\ - \beta \vm {\bar U}\inds\vm J(\vm x\inds)& \vm 0
	\end{bmatrix}\,.
\end{align}
According to \cite[Lem. 3.2.1]{Bertsekas2}, $\vm L(\vm p\inds) + \gamma \vm R(\vm p\inds)$ is positive definite for sufficiently large $\gamma > \bar \gamma$, since $\vm L(\vm p\inds)$ is positive definite on the nullspace of the positive semi-definite matrix $ \vm R(\vm p\inds) = \gamma \vm J(\vm x\inds)\trans (\vm{\bar U}\inds)^2 \vm J(\vm x\inds)$ by the second-order sufficiency condition in Assumption \ref{ass:SOSC}. Following the same arguments from \eqref{eq:partitioned_A_matrix} in the proof of Theorem \ref{th:conv} onward, we arrive at the quadratic eigenvalue problem 
\begin{align} \label{eq:modified_eigenvalues}
	(\lambda^2 \vm I - \vm {\tilde L}(\vm p\inds) \lambda  +\beta  (\vm J(\vm x\inds)\trans \vm {\bar U}\inds \vm J(\vm x\inds)))\vm v_x =\vm 0
\end{align} 
with  $\vm {\tilde L}(\vm p\inds) = \vm L(\vm p\inds) + \gamma \vm R(\vm p\inds)$. For \eqref{eq:modified_eigenvalues} it holds that $\mathrm{Re}(\lambda)>0$. This follows from the fact that  $\vm {\tilde L}(\vm p\inds)$ is positive definite for $\gamma > \bar \gamma$, $\beta \vm J(\vm x\inds)\trans \vm {\bar U}\inds \vm J(\vm x\inds)$ is positive semi-definite, and $\lambda \neq 0$. The proof of convergence then follows the proof of Theorem~\ref{th:conv} with \eqref{eq:partitioned_A_matrix_v2} instead of \eqref{eq:partitioned_A_matrix}. \hfill \QEDclosed

\subsection{Proof of Corollary \ref{cor:stricter_SOSC}}
\label{app:9}
Similar to \eqref{eq:preconditioning_matrix} in the proof of Lemma \ref{lem:approximation_v2}, we factorize the update law \eqref{eq:primal_dual_update_v3} from a centralized viewpoint to obtain
\begin{equation} \label{eq:preconditioning_matrix_v3}
	\vm P(\vm p\inds) = \begin{bmatrix}
		\vm I & \gamma \vm {\tilde J}_g(\vm x\inds)\trans & \gamma \vm {\tilde J}_h(\vm x\inds)\trans \vm U\inds  \\
		\vm 0 & -\beta \vm I & \vm 0 \\
		\vm 0 & \vm 0 &  -\beta \vm I
	\end{bmatrix} \vm M(\vm p\inds)
\end{equation}
with the partitioned Jacobians $\vm {\tilde J}_g(\vm x) = [ \nabla_{\vm x} \vm {\bar g}(\vm x)\trans, \vm 0\trans]\trans$ and $\vm {\tilde J}_h(\vm x) = [ \nabla_{\vm x} \vm {\bar h}(\vm x)\trans, \vm 0\trans]\trans$, derived from the update law \eqref{eq:primal_dual_update_v3}. Similar to \eqref{eq:partitioned_A_matrix_v2}, this leads to 
the modified sub-matrix 
\begin{align} \label{eq:partitioned_A_matrix_v3}
	\vm {\bar A} = \begin{bmatrix}
		\vm L(\vm p\inds) + \gamma \vm{\tilde J}(\vm x\inds)\trans (\vm {\bar U}\inds)^2 \vm{J}(\vm x\inds)& \vm J(\vm x\inds)\trans \\ - \beta \vm {\bar U}\inds \vm J(\vm x\inds)& \vm 0
	\end{bmatrix}
\end{align}
with  $\vm {\tilde J}(\vm x\inds) =: [\nabla_{\vm x}\trans \vm{ \bar g}(\vm x\inds), [\nabla_{\vm x} \vm{ \bar h}(\vm x\inds)]_{\mathcal{A}(\vm x\inds)}\trans, \vm 0\trans, \vm 0\trans]\trans$. Since $\vm {\bar U\inds}$ is diagonal with strictly positive values, we can evaluate
\begin{align}
&\vm{\tilde J}(\vm x\inds)\trans \vm{J}(\vm x\inds) = \\
& \nabla_{\vm x}\trans \vm{ \bar g}(\vm x\inds)\nabla_{\vm x} \vm{ \bar g}(\vm x\inds)+ [\nabla_{\vm x} \vm{ \bar h}(\vm x\inds)]_{\mathcal{A}(\vm x\inds)} \trans [\nabla_{\vm x} \vm{ \bar h}(\vm x\inds)]_{\mathcal{A}(\vm x\inds)} \nonumber
\end{align}
which is positive semi-definite. As in Theorem \ref{th:conv_v2}, \cite[Lem. 3.2.1]{Bertsekas2} implies that $\vm L(\vm p\inds) + \gamma \vm{\tilde J}(\vm x\inds)\trans (\vm {\bar U}\inds)^2 \vm{J}(\vm x\inds)$ is positive definite for sufficiently large $\gamma > \bar \gamma$, since $\vm L(\vm p\inds)$ is positive definite on the nullspace of $\nabla_{\vm x} \vm{ \bar g}(\vm x\inds)$ and $[\nabla_{\vm x} \vm{ \bar h}(\vm x\inds)]_{\mathcal{A}(\vm x\inds)}$ by Assumption \ref{ass:regularity_local_constraints}. The proof of convergence follows the proof of Theorem~\ref{th:conv} with \eqref{eq:partitioned_A_matrix_v3} instead of \eqref{eq:partitioned_A_matrix}. \hfill \QEDclosed
%%%%%%%%%%%%%%%%%%%%%%%%%%%%%%%%%%%%%%%%%%%%%%%%%%%%%%%%%%%%%%%%%%%%%%%%%%%%%%%%%%%%%%%%%%%%%%%%%%%%%%%%%%%%%%%%
\bibliographystyle{IEEEtran}
\bibliography{SBDP_plus_bib}
%%%%%%%%%%%%%%%%%%%%%%%%%%%%%%%%%%%%%%%%%%%%%%%%%%%%%%%%%%%%%%%%%%%%%%%%%%%%%%%%%%%%%%%%%%%%%%%%%%%%%%%%%%%%%%%%
\end{document}